\documentclass{imanum}

\usepackage{amsfonts,amsmath,amssymb}
\usepackage{mathrsfs,mathtools}
\usepackage{enumerate}
\usepackage{enumitem}
\usepackage{yfonts}
\usepackage{stackengine}
\usepackage{graphicx,psfrag}
\usepackage{scalerel}

\newcommand{\T}{\mathscr{T}}

\DeclareMathAlphabet{\mathpzc}{OT1}{pzc}{m}{it}

\usepackage[textsize=small]{todonotes}
\newcommand{\EO}[1]{{\color{black}#1}}
\newcommand{\EOP}[1]{{\color{black}#1}}

\jno{drnxxx}
\received{}
\revised{}

\begin{document}

\title{Semilinear optimal control with Dirac measures}
\shorttitle{Semilinear optimal control with Dirac measures}

\author{
{\sc Enrique Ot\'arola}
\thanks{enrique.otarola@usm.cl}
\\[2pt]
Departamento de Matem\'atica,
\\
Universidad T\'ecnica Federico Santa Mar\'ia, Valpara\'iso, Chile}
\shortauthorlist{E. Ot\'arola}

\maketitle

\begin{abstract}
{The purpose of this work is to study an optimal control problem for a semilinear elliptic partial differential equation \EO{with a linear combination of} Dirac measures \EO{as a forcing term}; the control variable corresponds to the amplitude of \EO{such singular sources}. We analyze the existence of optimal solutions and derive first and, necessary and sufficient, second order optimality conditions. We develop a solution technique that discretizes the state and adjoint equations with continuous piecewise linear finite elements; the control variable is already discrete. We analyze the convergence properties of discretizations and obtain, \EO{in two dimensions,} an a priori error estimate for the underlying approximation of an optimal control variable.}
{\EO{optimal control; semilinear equations; Dirac measures; first order optimality conditions; second order optimality conditions; finite element approximations; convergence; a priori error estimates.}}
\end{abstract}

\section{Introduction}\label{sec:introduction}
In this work we are \EO{concerned} with the analysis and discretization of an optimal control problem for a semilinear elliptic partial differential equation (PDE) \EO{with a linear combination of Dirac measures as a forcing term}. To make matters precise, we let \EO{$\Omega\subset\mathbb{R}^d$, with $d\in \{2, 3\}$,} be an open, bounded, and convex polytope with boundary $\partial\Omega$ and $\mathcal{D}$ be a finite ordered subset of $\Omega$ with cardinality $\ell:= \#\mathcal{D}<\infty$. Given a desired state $\mathsf{y}_d \in L^2(\Omega)$, a regularization parameter $\alpha>0$, and the cost functional
\begin{equation}\label{def:cost_func}
J(\mathsf{y},\mathbf{u}):=\frac{1}{2} \| \mathsf{y} - \mathsf{y}_d \|^2_{L^2(\Omega)}+ \frac{\alpha}{2} \| \mathbf{u} \|_{\mathbb{R}^{\ell}}^2,
\end{equation} 
we shall be concerned with the following PDE-constrained optimization problem: Find $\min J(\mathsf{y},\mathbf{u})$ subject to the \emph{monotone, semilinear,} and \emph{elliptic} PDE
\begin{equation}\label{def:state_eq}
-\Delta \mathsf{y} + a(\cdot,\mathsf{y})  =  \sum_{\mathsf{z} \in \mathcal{D}} \mathsf{u}_{\mathsf{z}} \delta_{\mathsf{z}}  \text{ in }  \Omega, 
\qquad
\mathsf{y}  =  0  \text{ on }  \partial\Omega,
\end{equation}
where $\delta_{\mathsf{z}}$ corresponds to the Dirac delta supported at the interior point $\mathsf{z} \in \mathcal{D}$, and 
\begin{equation}\label{def:box_constraints}
\mathbf{u} = \{\mathsf{u}_{\mathsf{z}} \}_{\mathsf{z} \in \mathcal{D}} \in \mathbb{R}^{\ell}: 
\quad 
\mathsf{a}_{\mathsf{z}} \leq \mathsf{u}_{\mathsf{z}} \leq \mathsf{b}_{\mathsf{z}}~\forall \mathsf{z} \in \mathcal{D}.  
\end{equation}
Here, $\mathbf{u}$ denotes the control variable. The control bounds $\mathbf{a} = \{\mathsf{a}_{\mathsf{z}} \}_{\mathsf{z} \in \mathcal{D}}$ and $\mathbf{b} = \{\mathsf{b}_{\mathsf{z}} \}_{\mathsf{z} \in \mathcal{D}}$ both belong to $\mathbb{R}^{\ell}$ and satisfy that $\mathsf{a}_{\mathsf{z}} < \mathsf{b}_{\mathsf{z}}$ for \EO{every} $\mathsf{z} \in \mathcal{D}$. Assumptions on the nonlinear function $a$ will be deferred until Section \ref{sec:assumption}.

\EO{PDE-constrained optimization problems involving measures have been previously} considered in a number of works. In particular, we mention the work by \cite{MR3162396}, in which the authors consider an optimal control problem where the state variable is governed by the semilinear elliptic equation \eqref{def:state_eq} but with a control variable $u$ that is measure \EO{valued, i.e., $u \in \mathcal{M}(\omega)$}, where $\omega \subset \Omega$. The main motivation behind this consideration is what the authors call \emph{sparsity promoting properties} of the control variable. In this framework, a complete analysis is provided for the state equation and the optimization problem. In particular, first and second order optimality conditions are derived. However, the work by \cite{MR3162396} is not concerned with approximation.

\EO{For the special situation of the linear case $a \equiv 0$}, there are a few works that consider the \EO{finite element discretization} of the optimal control problem previously introduced. In \cite{MR3225501}, the authors consider a discretization scheme based on the variational discretization approach and perform an a priori error analysis \cite[\EO{Theorem 3.7, item (i)}]{MR3225501}. The authors operate under the fact that the optimal state belongs to $W^{1,r}_0(\Omega)$ \EO{for every $r \in (1, d/(d-1))$.} In contrast, an approach using Muckenhoupt weights, weighted Sobolev spaces, and the corresponding weighted norm inequalities has been explored in \cite{MR3800041}. In such a work, the authors obtain the following convergence rates for the error approximation of the optimal control variable: $\mathcal{O}(h^{2-\epsilon})$ in two dimensions and $\mathcal{O}(h^{1-\epsilon})$ in three dimensions, where $\epsilon > 0$ is arbitrarily small \cite[Theorem 5.1]{MR3800041}; the error estimate in two dimensions being improved in Theorem \ref{thm:error_estimate_control_final} below to $\mathcal{O}(h^{2}|\log h|^3)$. The extension of such weighted techniques for the case where the state equations are the Stokes equations has been considered in \cite{FOQ}. Related optimal control problems within the context of the active control of sound and vibrations have been studied in \cite{MR2086168} and \cite{MR2525606}, respectively. Finally, we refer to the references \cite{MR2974716}, \cite{MR3072225}, and \cite{MR3945081}, where the authors study discretization techniques for \EO{an optimal control} problem without control constraints, but where the control is a regular Borel measure.

\EO{Apart from the fact that} this exposition is the first to study a numerical scheme for the semilinear optimal control problem $\min J(\mathsf{y},\mathbf{u})$ subject to \eqref{def:state_eq} and \eqref{def:box_constraints}, the analysis itself comes with its own set of difficulties. \EO{Below we list what we consider to be the most important contributions of our work:}
\begin{itemize}
\item[(i)] \emph{Error estimates for semilinear PDEs with \EO{Borel} measures:} For a basic finite element discretization of \EO{\eqref{eq:semilinear_PDE}}, we obtain a $L^2$-error estimate that is optimal in terms of regularity (Theorem \ref{thm:eq:error_estimate_L2_semilinear}). We also derive a nearly--optimal, in terms of approximation, $L^1$-error estimate (Theorem \ref{thm:L1_error_estimate}).

\item[(ii)] \emph{\EO{Existence and optimality conditions}:} \EO{We show the existence of at least one global solution to our optimal control problem; see Theorem \ref{thm:existence_optimal_pair}. Moreover, we analyze first and, necessary and sufficient, second order optimality conditions in Sections \ref{sec:first_order_o_c} and \ref{sec:second_order_o_c}, respectively.}

\item[(iii)] \emph{Convergence of discretizations:} \EO{We prove the existence of subsequences of global solutions of suitable discrete problems that converge to a global solution of the continuous optimal control problem (Theorem \ref{thm:convergence_control}). We also prove that continuous strict local solutions can be approximated by local minima of the aforementioned discrete problems (Theorem \ref{thm:convergence_control_2}).}

\item[(iv)] \emph{Error estimates:} \EO{We derive in Theorem \ref{thm:error_estimate_control_final} an error estimate for the underlying approximation of an optimal control variable in two dimensions that is nearly--optimal in terms of approximation.}
\end{itemize}

\EO{We organize our presentation as follows}. \EO{In Section \ref{sec:nota_and_assum} we present notation and gather some facts that shall be useful for our purposes. In particular, we briefly review well-posedness results for semilinear elliptic PDEs with singular forcing. Section \ref{sec:finite_element_approximation} is our first original contribution. We obtain a $L ^2$-error estimate for a standard finite element approximation of the aforementioned semilinear elliptic PDEs.} Section \ref{sec:semilinear_optimal_control} is dedicated to the analysis of the semilinear optimal control problem. In particular, we derive first and, necessary and sufficient, second order optimality conditions. \EO{Finally, in Section \ref{sec:fem_ocp}, we devise and analyze a suitable finite element discretization scheme for our semilinear optimal control problem: we derive convergence results and obtain, in two dimensions, a priori error estimates.}


\section{Notation, assumptions, and preliminaries}
\label{sec:nota_and_assum}

Let us fix notation and the setting in which we will operate. Throughout this work, $d \in \{2,3\}$ and $\Omega$ is an open, bounded, and convex polytopal domain in $\mathbb{R}^d$. We denote by $\partial \Omega$ the boundary of $\Omega$. 

If $\mathscr{X}$ and $\mathscr{Y}$ are Banach function spaces, we write $\mathscr{X}\hookrightarrow \mathscr{Y}$ to denote that $\mathscr{X}$ is continuously embedded in $\mathscr{Y}$. We denote by $\mathscr{X}'$ and $\| \cdot \|_{\mathscr{X}}$ the dual and the norm of $\mathscr{X}$, respectively. We denote by $\langle \cdot, \cdot \rangle_{\mathscr{X}', \mathscr{X}}$ the duality paring between $\mathscr{X}'$ and $\mathscr{X}$; we shall simply denote $\langle \cdot, \cdot \rangle$ whenever the underlying spaces are clear from the context. Let \EO{$\{ x_n \}_{n \in \mathbb{N}}$} be a sequence in $\mathscr{X}$. We will denote by $x_n \rightarrow x$ and $x_n \rightharpoonup x$ the  strong and weak convergence, respectively, of $\{ x_n \}_{n \in \mathbb{N}}$ to $x$. We will use standard notation for Sobolev spaces, norms, and seminorms. We denote by $\mathcal{M}(\Omega)$ the space of Radon measures on $\Omega$ -- the space of regular Borel measures $\mu$ which are such that $\mu(K) < \infty$ for every compact set $K \subset \EO{\mathbb{R}^d}$ \cite[\EO{page 5}]{MR1158660} -- and recall that $\mathcal{M}(\Omega)$ can be identified with $C_0(\Omega)'$ -- the dual of the space of continuous functions \EO{in} $\bar \Omega$ vanishing on $\partial \Omega$ \cite[Theorem 7.17]{MR1681462}. 

Given $s \in (1,\infty)$, we denote by $s'$ its H\"older conjugate, i.e., the real number such that $1/s + 1/s' = 1$. The relation $A \lesssim B$ indicates that $A \leq C B$, with a positive constant \EO{that does not depend on $A$, $B$, or} the underlying discretization parameters. The value of $C$ might change at each occurrence.


\subsection{Assumptions}\label{sec:assumption}

\EO{We will operate under the following assumptions on $a$; see \cite[Section 2]{MR3162396}.} We must, however, immediately mention that some of the results obtained in this work are valid under less restrictive requirements; when possible we explicitly mention the assumptions on the nonlinear term $a$ that are required for a particular result. 

\begin{enumerate}[label=(A.\arabic*)]
\item \label{A1} $a  = a(x,y):\Omega \times \mathbb{R} \rightarrow \mathbb{R}$ is a Carath\'eodory function that is monotone increasing in $y$ for a.e.~$x$ in $\Omega$ and satisfies the growth condition
\begin{equation}
\label{eq:growing_a}
|a(x,y)| \leq |\phi_0(x)| + C_a |y|^r
\quad
\mathrm{a.e.}~x \in \Omega,~\forall y \in \mathbb{R}.
\end{equation} 
Here, $C_a >0$ is a constant, $\phi_0 \in L^1(\Omega)$, $r < \infty$ if $d=2$, and $r<3$ if $d=3$.
\item \label{A2} $a = a(x,y):\Omega \times \mathbb{R} \rightarrow \mathbb{R}$ is a Carath\'eodory function of class $C^1$ with respect to $y$ for a.e.~$x$ in $\Omega$ and there exists $\phi_1 \in L^{q}(\Omega)$, with $q>d/2$, such that
\begin{equation}
\label{eq:growing_a_derivative}
0 \leq \tfrac{\partial a}{\partial y}(x,y) \leq |\phi_1(x)| + C_a |y|^r
\quad
\mathrm{a.e.}~x \in \Omega,~\forall y \in \mathbb{R}.
\end{equation} 
Here, $C_a > 0$ is a constant, $r < \infty$ if $d=2$, and $r<2$ if $d=3$.
%
%
%
%
\item \label{A3} $a = a(x,y):\Omega \times \mathbb{R} \rightarrow \mathbb{R}$ is a Carath\'eodory function of class $C^{2}$ with respect to $y$ for a.e.~$x$ in $\Omega$ and there exists $\phi_2 \in L^{t}(\Omega)$ such that
\begin{equation}
\label{eq:growing_a_second_derivative}
\left| \tfrac{\partial^2 a}{\partial y^2}(x,y) \right| \leq |\phi_2(x)| + C_a |y|^r
\quad
\mathrm{a.e.}~ x \in \Omega,~\forall y \in \mathbb{R},
\end{equation} 
where $t>1$ if $d=2$ and $t>3$ if $d=3$, $C_a > 0$ is a constant, $r < \infty$ if $d=2$, and $r<1$ if $d=3$.
\end{enumerate}

\EO{The following comments are now in order. We first notice that our assumptions \ref{A1}, \ref{A2}, and \ref{A3} are those stated in (A1), (A2), and (A3), respectively, in the work by \cite{MR3162396}. Second, if \ref{A2} holds and $a(\cdot,0) \in L^1(\Omega)$, then \ref{A1} is satisfied. Third, if \ref{A3} holds, $\partial a/\partial y(\cdot,0) \in L^{q}(\Omega)$, for some $q>d/2$, and $\partial a/\partial y(x,y) \geq 0$ for a.e.~$x$ in $\Omega$ and $y \in \mathbb{R}$, then \ref{A2} holds.}

Further assumptions on $a$ that will be particularly needed for \EO{performing an a priori error analysis for a suitable finite element discretization} will be deferred until Section \ref{sec:fem_ocp}.


\section{\EO{Semilinear PDEs with singular forcing}}
\label{sec:state_equation}
\EO{Let $\Omega$ be an open and bounded domain with Lipschitz boundary, $\mu \in \mathcal{M}(\Omega)$, and $a$ be such that \ref{A1} holds. Let us introduce the following semilinear elliptic PDE with singular forcing}: Find $y$ such that
\begin{equation}
- \Delta y + a(\cdot,y) = 
\EO{\mu}
\text{ in }  \Omega, \qquad
y =  0  \text{ on }  \partial\Omega.
\label{eq:semilinear_PDE}
\end{equation}
The following notion of weak solution follows from \cite[Section 2]{MR3162396}: $y \in L^1(\Omega)$ is a weak solution for problem \eqref{eq:semilinear_PDE} if $a(\cdot,y) \in L^1(\Omega)$ and
\begin{equation}
\int_{\Omega} \left[ - y \Delta v + a(x,y) v \right] \mathrm{d}x = 
\EO{\int_{\Omega} v \mathrm{d}\mu}
\qquad \forall v \in \mathcal{Z}(\Omega).
\label{eq:weak_semilinear_PDE}
\end{equation}
Here, $\mathcal{Z}(\Omega) := \{ v \in H_0^1(\Omega): \Delta v \in C(\bar \Omega) \}$. We immediately note that $\mathcal{Z}(\Omega) \subset C_0 (\Omega)$ and observe that this \EO{property and} the definition of $\mathcal{Z}(\Omega)$ imply that all the terms involved in \eqref{eq:weak_semilinear_PDE} are well defined.

The following result states the well-posedness of \EO{\eqref{eq:weak_semilinear_PDE}} and further regularity properties for the solution $y$; \EO{see \cite[Theorem 3]{MR1025884} and} \cite[Theorem 2.1]{MR3162396}.

\begin{theorem}[well-posedness]
\EO{There exists a unique solution to \eqref{eq:weak_semilinear_PDE}. In addition, $y \in W_0^{1,p}(\Omega)$ and}
\begin{equation}\label{eq:stability_state}
\| \nabla y \|_{L^p(\Omega)}
\lesssim
\EO{\| \mu \|_{\mathcal{M}(\Omega)}}
+
\|a(\cdot,0)\|_{L^{1}(\Omega)} 
\end{equation}
\EO{for every $p < d/(d-1)$.} The hidden constant is independent of $y$, $a$, \EO{and $\mu$.}
\label{thm:well_posedness_semilinear}
\end{theorem}

The following remark is in order.

\begin{remark}[variational formulation]
Since the solution $y$ to problem \eqref{eq:weak_semilinear_PDE} belongs to $W_0^{1,p}(\Omega)$, for every $p<d/(d-1)$, the following alternative weak formulation for problem \eqref{eq:semilinear_PDE} can be formulated \cite[Remark 2.3]{MR3162396}: 
\begin{equation}
\EO{y \in W_0^{1,p}(\Omega):}
\quad
 \int_{\Omega} \nabla y \cdot \nabla v \mathrm{d}x + \int_{\Omega} a(x,y) v \mathrm{d}x = \EO{\int_{\Omega} v \mathrm{d}\mu} \quad \forall v \in W_0^{1,p'}(\Omega).
\label{eq:weak_semilinear_PDE_W1p}
\end{equation}
Here, $p'>d$ denotes the H\"older conjugate of $p$. Notice that, within the considered functional space setting, all the terms involved in the weak formulation \EO{\eqref{eq:weak_semilinear_PDE_W1p}} are well defined.
\label{rem:variational_formulation}
\end{remark}


\subsection{\EO{Finite element discretization: an error estimate in $L^2(\Omega)$}}
\label{sec:finite_element_approximation}

\EO{In this section, we shall assume, in addition, that $\Omega$ is a \emph{convex} polytope so that it can be triangulated exactly. We thus introduce} $\mathscr{T}_h = \{ T\}$, a conforming partition of $\bar{\Omega}$ into closed simplices $T$ with size $h_T = \text{diam}(T)$, and define \EO{$h:=\max \{  h_T: T \in \mathscr{T}_h \}$}. We denote by $\mathbb{T} = \{\mathscr{T}_h \}_{h>0}$ a collection of conforming and quasi-uniform meshes $\mathscr{T}_h$, which are refinements of a common mesh $\mathscr{T}_{\star}$. Given a mesh $\mathscr{T}_{h} \in \mathbb{T}$, we define the finite element space of continuous piecewise polynomials of degree one as
\begin{equation}\label{eq:piecewise_linear_set}
\mathbb{V}_{h}:=\{v_{h}\in C(\bar{\Omega}): v_{h}|_T\in \mathbb{P}_{1}(T) \ \forall T\in \T_{h}\}\cap H_0^1(\Omega).
\end{equation}

We define the Galerkin approximation of the solution $y$ to problem \eqref{eq:weak_semilinear_PDE_W1p} by
\begin{equation}\label{eq:weak_semilinear_discrete}
y_h \in\mathbb{V}_h : 
\quad
\int_{\Omega} \nabla y_h \cdot \nabla v_{h} \mathrm{d}x + \int_{\Omega} a(x,y_h) v_{h} \mathrm{d}x = 
\EO{\int_{\Omega}v_h \mathrm{d}\mu}
\quad \forall v_h \in \mathbb{V}_h.
\end{equation}
\EO{The existence of a discrete solution $y_h$, for a fixed mesh $\T_h$, follows from Brouwer's fixed point theorem; uniqueness follows from the monotonicity of $a$.}

We derive an error estimate in $L^2(\Omega)$ that is optimal with respect to regularity. This estimate, which is of independent interest, extends the linear theory developed by \cite{MR812624} to a semilinear scenario.

\begin{theorem}[\EO{$L^2(\Omega)$-error estimate}]
Let $a = a(x,y) : \Omega \times \mathbb{R} \rightarrow \mathbb{R}$ be a Carath\'eodory function that is monotone increasing in $y$ \EO{for} a.e.~$x \in \Omega$ and satisfies \eqref{eq:growing_a}. Assume, in addition, that 
\begin{equation}
|a(x,y) - a(x,z)| \leq |\psi(x)| | y-z | ~\mathrm{a.e.}~x \in \Omega,~\forall y,z \in \mathbb{R}, 
\qquad
\psi \in L^s(\Omega),
\label{eq:a_Lipschitz}
\end{equation}
\EOP{where $s=2$.} If $h$ is sufficiently small, \EO{then} we have the optimal error estimate
\begin{equation}
\label{eq:error_estimate_L2_semilinear}
\| y - y_h \|_{L^2(\Omega)} \lesssim h^{2-\frac{d}{2}},
\end{equation}
with a hidden constant that is independent of $y$, $y_h$, and $h$.
\label{thm:eq:error_estimate_L2_semilinear}
\end{theorem}
\begin{proof}
We proceed on the basis of a duality argument \EO{and} begin the proof by introducing the \emph{nonnegative} function 
\EO{$\chi$ as follows:}
\[
\chi(x) = \frac{a(x,y(x)) - a(x,y_h(x))}{y(x) - y_h(x)}
~\mathrm{if}~y(x) \neq y_h(x),
\quad
\chi(x) = 0 
~\mathrm{if}~y(x) = y_h(x).
\]
\EO{Observe that \eqref{eq:a_Lipschitz} guarantees that $\chi \in L^s(\Omega)$. Let $\mathfrak{f} \in L^2(\Omega)$ \EO{and let} $\mathfrak{z}$ and $\mathfrak{z}_h$ be the solutions to 
\begin{align}
 (\nabla v, \nabla \mathfrak{z})_{L^2(\Omega)} + (\chi \mathfrak{z}, v)_{L^2(\Omega)} & =  (\mathfrak{f},v)_{L^2(\Omega)} \qquad 
 \forall v \in H_0^1(\Omega),
 \label{eq:problem_z}
 \\
 (\nabla v_h, \nabla \mathfrak{z}_h)_{L^2(\Omega)} + (\chi \mathfrak{z}_h, v_h)_{L^2(\Omega)} & = (\mathfrak{f},v_h)_{L^2(\Omega)} \qquad 
 \forall v_h \in \mathbb{V}_h,
 \label{eq:problem_zh}
\end{align}
respectively. \EOP{Since $\Omega$ is Lipschitz, $\chi \in L^q(\Omega)$, for some $q>d/2$, and $\mathfrak{f} \in L^2(\Omega)$, we invoke \cite[Theorem 4.2]{MR192177} to deduce that $\mathfrak{z} \in L^{\scaleto{\infty}{4pt}}(\Omega)$. As a result, we have that $\mathfrak{f} - \chi\mathfrak{z} \in L^{2}(\Omega)$. We are thus in position to invoke standard elliptic regularity theory on the basis of the convexity of $\Omega$ to deduce that $\mathfrak{z} \in H^2(\Omega) \cap H_0^1(\Omega)$. Notice that, in view of the Sobolev embedding in \cite[Theorem 4.12, Part I, Case C]{MR2424078}, we deduce the existence of $r>d$ such that $\mathfrak{z} \in W_{\scaleto{0}{4pt}}^{\scaleto{1}{4pt},\scaleto{r}{3.5pt}}(\Omega)$.} Let us also observe that
\begin{equation}
(\nabla v_h, \nabla (\mathfrak{z}-\mathfrak{z}_h))_{L^2(\Omega)} + (\chi (\mathfrak{z}-\mathfrak{z}_h), v_h)_{L^2(\Omega)} = 0 \quad \forall v_h \in \mathbb{V}_h.
\label{eq:z-z_h}
\end{equation}
We now bound $\| \mathfrak{z}-\mathfrak{z}_h \|_{L^{\infty}(\Omega)}$. Let $I_h: H^2(\Omega) \cap H_0^1(\Omega) \rightarrow \mathbb{V}_h$  be the Lagrange interpolation operator. A standard interpolation error estimate combined with a finite element error estimate in $L^2(\Omega)$ yield
\[
 \| I_h \mathfrak{z} - \mathfrak{z}_h \|_{L^{2}(\Omega)} \leq \| I_h \mathfrak{z} - \mathfrak{z} \|_{L^{2}(\Omega)} + \| \mathfrak{z} - \mathfrak{z}_h \|_{L^{2}(\Omega)} \lesssim h^2|\mathfrak{z} |_{H^2(\Omega)}.
\]
Let us now utilize an inverse inequality and $\| \mathfrak{z}-I_h \mathfrak{z} \|_{L^{\infty}(\Omega)} \lesssim h^{\sigma} |\mathfrak{z} |_{H^2(\Omega)}$, where $\sigma = 2 - d/2$, to obtain}
\begin{equation*}
 \EO{\| \mathfrak{z}-\mathfrak{z}_h \|_{L^{\infty}(\Omega)} \leq  \| \mathfrak{z}-I_h \mathfrak{z} \|_{L^{\infty}(\Omega)}
 +
 \| I_h \mathfrak{z} - \mathfrak{z}_h \|_{L^{\infty}(\Omega)}
 \lesssim 
 h^{\sigma} |\mathfrak{z} |_{H^2(\Omega)}
 +
 h^{-\frac{d}{2}} \| I_h \mathfrak{z} - \mathfrak{z}_h \|_{L^{2}(\Omega)}
 \leq 
 Ch^{\sigma}\|\mathfrak{f} \|_{L^2(\Omega)},}
\end{equation*}
\EO{where $C$ depends on $\| \chi\|_{L^s(\Omega)}$. This error estimate shows, in particular, that $\mathfrak{z}_h \in L^{\infty}(\Omega)$ uniformly with respect to discretization; recall that $\mathfrak{z} \in H^2(\Omega) \cap H_0^1(\Omega) \hookrightarrow C(\bar \Omega)$.}

Set $\mathfrak{f} = y - y_h \in L^2(\Omega)$. A density argument allows us to set $v = y - y_h$ in problem \eqref{eq:problem_z} and obtain
\begin{align*}
\| y-y_h \|^2_{L^2(\Omega)}   
 & = (\nabla (y-y_h), \nabla \mathfrak{z})_{L^2(\Omega)} + (\chi \mathfrak{z}, y - y_h)_{L^2(\Omega)} 
\\
 & =
(\nabla (y-y_h), \nabla (\mathfrak{z}-\mathfrak{z}_h))_{L^2(\Omega)} 
+
(a(\cdot,y)-a(\cdot,y_h), \mathfrak{z}-\mathfrak{z}_h)_{L^2(\Omega)}.
\end{align*}
We stress that, since there exists $r>d$ such that $\mathfrak{z} \in W_0^{1,r}(\Omega)$ and $y-y_h \in W_0^{1,p}(\Omega)$, for every $p<d/(d-1)$, \EO{the term} $(\nabla (y-y_h), \nabla (\mathfrak{z}-\mathfrak{z}_h))_{L^2(\Omega)}$ is well defined. \EO{Now,} since $y$ solves \eqref{eq:weak_semilinear_PDE_W1p}, we obtain 
\begin{multline}
\| y-y_h \|^2_{L^2(\Omega)} =
\EO{\int_{\Omega} (\mathfrak{z}-\mathfrak{z}_h) \mathrm{d}\mu} - (\nabla y_h, \nabla (\mathfrak{z}-\mathfrak{z}_h))_{L^2(\Omega)} 
\\
- 
\EO{(a(\cdot,y),\mathfrak{z}-\mathfrak{z}_h)_{L^2(\Omega)}
+
(a(\cdot,y)-a(\cdot,y_h),\mathfrak{z}-\mathfrak{z}_h)_{L^2(\Omega)}},
\label{eq:aux_L2_estimate}
\end{multline}
upon utilizing that there exits $r>d$ such that $\mathfrak{z} \in W_0^{1,r}(\Omega)$ \EO{and that $W_0^{1,r}(\Omega) \hookrightarrow C(\bar \Omega)$.} We thus set $v_h = y_h$ in \EO{\eqref{eq:z-z_h} to obtain, in view of the identity \eqref{eq:aux_L2_estimate}, the relation}
\begin{multline}
\| y-y_h \|^2_{L^2(\Omega)} = \EO{\int_{\Omega} (\mathfrak{z}-\mathfrak{z}_h)\mathrm{d}\mu} + (\chi (\mathfrak{z}-\mathfrak{z}_h), y_h)_{L^2(\Omega)} 
\\
- 
\EO{(a(\cdot,y),\mathfrak{z}-\mathfrak{z}_h)_{L^2(\Omega)}
+
(a(\cdot,y)-a(\cdot,y_h),\mathfrak{z}-\mathfrak{z}_h)_{L^2(\Omega)}}.
\end{multline}
\EO{Since $\mathfrak{f} = y - y_h$, this identity and the error bound $\| \mathfrak{z}-\mathfrak{z}_h \|_{L^{\infty}(\Omega)} \lesssim h^{\sigma}\|\mathfrak{f} \|_{L^2(\Omega)}$ imply the estimate}
\begin{equation}
\| y-y_h \|_{L^2(\Omega)} \lesssim h^{\sigma}\left( 
\| \mu \|_{\mathcal{M}(\Omega)} 
+ \| \chi y_h\|_{L^1(\Omega)}  +  
\EO{\| a(\cdot,y) \|_{L^1(\Omega)}
+
\| \Psi \|_{L^2(\Omega)} \| y - y_h \|_{L^2(\Omega)}}
\right).
\label{eq:L2_estimate_aux}
\end{equation}
\EO{In view of \eqref{eq:growing_a} and \eqref{eq:stability_state}, we deduce $\| a(\cdot,y) \|_{L^1(\Omega)} \lesssim  \| \phi_0 \|_{L^1(\Omega)} + [ \| \mu \|_{\mathcal{M}(\Omega)} + \| a(\cdot,0) \|_{L^1(\Omega)}]^r$. The control of $\| \chi y_h\|_{L^1(\Omega)}$ follows from H\"older's inequality, the triangle inequality and \eqref{eq:stability_state}:}
\begin{equation}
 \EO{\| \chi y_h\|_{L^1(\Omega)}  
 \lesssim \|\chi\|_{L^2(\Omega)} \left[ \| \mu \|_{\mathcal{M}(\Omega)} + \| a(\cdot,0) \|_{L^1(\Omega)} + \| y - y_h \|_{L^2(\Omega)}\right].}
 \label{eq:L2_aux_aux}
\end{equation}
\EO{Replace the estimates obtained for $\| \chi y_h\|_{L^1(\Omega)}$ and $\|a(\cdot,y)\|_{L^1(\Omega)}$} into \eqref{eq:L2_estimate_aux} and utilize the assumption that $h$ is sufficiently small so that the terms involving $\|y-y_h\|_{L^2(\Omega)}$ in \eqref{eq:L2_estimate_aux} and \eqref{eq:L2_aux_aux} can be absorbed in the left hand side of \eqref{eq:L2_estimate_aux}. These arguments yield the desired bound \eqref{eq:error_estimate_L2_semilinear} and conclude the proof.
\end{proof}


\section{The semilinear optimal control problem with singular sources}
\label{sec:semilinear_optimal_control}
Let us precisely describe and analyze the \EO{semilinear} optimal control problem with point sources introduced in Section \ref{sec:introduction}. We begin our studies by defining the set of admissible controls 
\begin{equation}
\label{eq:Uad}
\mathbf{U}_{ad}:= \left\{ \mathbf{u} = \{\mathsf{u}_{\mathsf{z}} \}_{\mathsf{z} \in \mathcal{D}} \in \mathbb{R}^{\ell}
:
~\mathsf{a}_{\mathsf{z}} \leq \mathsf{u}_{\mathsf{z}} \leq \mathsf{b}_{\mathsf{z}}~\forall \mathsf{z} \in \mathcal{D} \right \},
\end{equation}
where $\mathbf{a} = \{ \mathsf{a}_{\mathsf{z}} \}_{\mathsf{z} \in \mathcal{D}}$ and $\mathbf{b} = \{ \mathsf{b}_{\mathsf{z}} \}_{\mathsf{z} \in \mathcal{D}}$ both belong to $\mathbb{R}^{\ell}$ and satisfy that $\mathsf{a}_{\mathsf{z}} < \mathsf{b}_{\mathsf{z}}$ for \EO{every} $\mathsf{z} \in \mathcal{D}$. We recall that $\mathcal{D}$ denotes a finite ordered subset of $\Omega$ with cardinality $\# \mathcal{D} = \ell$. Notice that $\mathbf{U}_{ad}$ \EO{is a nonempty, closed, and bounded} subset of $\mathbb{R}^{\ell}$. 

We define the \emph{semilinear optimal control problem with point sources} as follows: Find
\begin{equation}
\label{eq:min}
\min \{ J (\mathsf{y},\mathbf{u}): (\mathsf{y},\mathbf{u}) \in W_0^{1,p}(\Omega) \times \mathbf{U}_{ad} \}
\end{equation}
subject to the following weak formulation of the state equation: Find $\mathsf{y} \in W_0^{1,p}(\Omega)$ such that
\begin{equation}
\label{eq:state_equation}
 \int_{\Omega} \nabla \mathsf{y} \cdot \nabla v \mathrm{d}x + \int_{\Omega} a(x,\mathsf{y}) v \mathrm{d}x = \sum_{\mathsf{z} \in \mathcal{D} } \mathsf{u}_{\mathsf{z}} v(\mathsf{z}) 
 \quad 
 \forall v \in W_0^{1,p'}(\Omega),
 \quad
 p<d/(d-1),
 \quad 
 p'>d.
\end{equation}
Here, $p'$ denotes the H\"older conjugate of $p$. \EO{Define $\mathfrak{G}: W_0^{1,p'}(\Omega) \rightarrow \mathbb{R}$ by $\mathfrak{G} v = \sum_{\mathsf{z} \in \mathcal{D}} \mathsf{u}_{\mathsf{z}} v(\mathsf{z})$ and observe that $\mathfrak{G} \in \mathcal{M}(\Omega)$. Let us now} assume that $a = a(x,y) : \Omega \times \mathbb{R} \rightarrow \mathbb{R}$ is a Carath\'eodory function that is monotone increasing in $y$ for a.e.~$x$ in $\Omega$ and satisfies \eqref{eq:growing_a}. \EO{In view of this assumption and the fact that $\mathfrak{G} \in \mathcal{M}(\Omega)$, an application of Theorem \ref{thm:well_posedness_semilinear} yields the existence of a unique solution $\mathsf{y}$ to problem \eqref{eq:state_equation}}.

\EO{In order to analyze our optimal control problem, we introduce the so-called control to state map} $\mathcal{S}: \mathbf{U}_{ad} \rightarrow W_{\scaleto{0}{4pt}}^{1,p}(\Omega)$, which, given a control \EO{$\mathbf{u} \in \mathbf{U}_{ad}$}, associates to it the unique state \EO{$\mathsf{y} \in W_{\scaleto{0}{4pt}}^{1,p}(\Omega)$} that solves \eqref{eq:state_equation}. We notice that the map $\mathcal{S}$ is bounded. In fact, Theorem \ref{thm:well_posedness_semilinear} immediately yields the bound 
\[
\| \nabla (\mathcal{S} \mathbf{u}) \|_{L^{p}(\Omega)} \lesssim 
\EO{\sum_{\mathsf{z} \in \mathcal{D}} |\mathsf{u}_{\mathsf{z}}|
+
\| a(\cdot,0)\|_{L^1(\Omega)}},
\qquad
p<d/(d-1).
\]
With $\mathcal{S}$ at hand, we introduce 
$j: \mathbf{U}_{ad} \rightarrow \mathbb{R}$ by $j(\mathbf{u}):=J(\mathcal{S}\mathbf{u},\mathbf{u})$. \EO{We recall that $J$ is defined in \eqref{def:cost_func}.}


\subsection{Existence of optimal controls}
\label{sec:existence_optimal_control}

The existence of an optimal state-control pair \EO{$(\bar{\mathsf{y}},\bar{\mathbf{u}}) \in W_0^{1,p}(\Omega) \times \mathbf{U}_{ad}$ is as follows.}

\begin{theorem}[\EO{existence}]
Let $\Omega$ be an open and bounded domain with Lipschitz boundary \EO{and $a$ be such that \ref{A1} holds.}
Then, problem \eqref{eq:min}--\eqref{eq:state_equation} admits at least one global solution $(\bar{\mathsf{y}},\bar{\mathbf{u}})$.
\label{thm:existence_optimal_pair}
\end{theorem}
\begin{proof}
\EO{Using the reduced cost functional $j$, our optimal control problem \eqref{eq:min}--\eqref{eq:state_equation} reduces to: Minimize $j(\mathbf{u})$ over $\mathbf{U}_{ad}$. Let us now observe that, as a bounded and closed set in a finite dimensional space, $\mathbf{U}_{ad}$ is compact. On the other hand, the convergence results provided in \cite[Theorem 2.1]{MR3162396} guarantee that $j$ is continuous on $\mathbf{U}_{ad}$. With these ingredients at hand, the desired result follows from the well known Weierstrass theorem.}
\end{proof}


\subsection{\EO{Differentiability properties of the control to state map $\mathcal{S}$ and the adjoint equation}}
\label{sec:1st_order}
In this section, we analyze differentiability properties of $\mathcal{S}$ and derive first order necessary optimality conditions for \eqref{eq:min}--\eqref{eq:state_equation}. Since the control problem \eqref{eq:min}--\eqref{eq:state_equation} is not convex, we discuss optimality conditions in the context of local solutions.

\subsubsection{\EO{Differentiability of the control to state map}} We present the following result.

\begin{theorem}[\EO{differentiability of $\mathcal{S}$}]
\label{thm:first_order_diff}
\EO{Let $\Omega$ be an open and bounded domain with Lipschitz boundary. If $a$ is such that \textnormal{\ref{A1} and \ref{A2}} hold, then the map $\mathcal{S}: \mathbb{R}^{\ell} \rightarrow W_0^{1,p}(\Omega)$ is of class $C^1$ for $p<d/(d-1)$. 
In addition, if $\mathbf{u}, \mathbf{v} \in \EO{\mathbb{R}^{\ell}}$, then $\phi = \mathcal{S}'(\mathbf{u}) \mathbf{v} \in W_0^{1,p}(\Omega)
$ corresponds to the unique solution to 
\begin{equation}\label{eq:aux_adjoint}
(\nabla \phi, \nabla v)_{L^2(\Omega)} +\left(\tfrac{\partial a}{\partial y}(\cdot,\mathsf{y})\phi, v \right)_{L^2(\Omega)}=
\sum_{\mathsf{z} \in \mathcal{D}} \mathsf{v}_{\mathsf{z}} \langle \delta_{\mathsf{z}}, v \rangle 
\quad 
\forall v \in W_0^{1,p'}(\Omega),
\quad
\mathsf{y} = \mathcal{S}(\mathbf{u}).
\end{equation}
If, in addition, \textnormal{\ref{A3}} holds, then $\mathcal{S}: \mathbb{R}^{\ell} \rightarrow W_0^{1,p}(\Omega)$ is of class $C^2$ for $p<d/(d-1)$. Let $\mathbf{u}, \mathbf{v}, \mathbf{w} \in \mathbb{R}^{\ell}$. Then, $\varphi = \mathcal{S}''(\mathbf{u})( \mathbf{v},\mathbf{w}) \in W_0^{1,p}(\Omega)$ corresponds to the unique solution to
\begin{equation}\label{eq:aux_second_adjoint_new}
(\nabla \varphi, \nabla v)_{L^2(\Omega)} +\left(\tfrac{\partial a}{\partial y}(\cdot,\mathsf{y})\varphi, v \right)_{L^2(\Omega)}=
- \left( \tfrac{\partial^2 a}{\partial y^2}(\cdot,\mathsf{y}) \phi_{\mathbf{v}} \phi_{\mathbf{w}}, v \right)_{L^2(\Omega)}
\quad
\forall v \in W_0^{1,p'}(\Omega),
\quad \!\!
\mathsf{y} = \mathcal{S}(\mathbf{u}).
\end{equation}
Here, $\phi_{\mathbf{v}} = \mathcal{S}'(\mathbf{u}) \mathbf{v}$ and $\phi_{\mathbf{w}} = \mathcal{S}'(\mathbf{u}) \mathbf{w}$.} 
\end{theorem}
\begin{proof}
\EO{We begin the proof by noticing that $\mathsf{y} \mapsto a(\cdot,\mathsf{y})$ is $C^1$ as a map from $W_0^{1,p}(\Omega)$ into $L^1(\Omega)$ for $p<d/(d-1)$. This result follows from an an adaption of the arguments elaborated in the proof of \cite[Lemma 4.12]{Troltzsch} combined with the fact that $\partial a/ \partial y$ satisfies \eqref{eq:growing_a_derivative}. The first order differentiability of $\mathcal{S}$ thus follows from \cite[Theorem 2.2]{MR3162396}. Alternatively, such a result can be derived utilizing the arguments elaborated in the proof of \cite[Theorem 4.17]{Troltzsch}. 

Let us now observe that, since \ref{A3} holds, $W_0^{1,p}(\Omega) \ni \mathsf{y} \mapsto a(\cdot,\mathsf{y}) \in L^1(\Omega)$ is of class $C^2$. The second order differentiability of $\mathcal{S}$ thus follows from \cite[Theorem 2.2]{MR3162396}.} 
\end{proof}

\subsubsection{The adjoint equation}
\EO{In order to formulate} first order optimality conditions, we present a classical result: If $\mathbf{\bar{u}} \in \mathbf{U}_{ad}$ denotes a locally optimal control, then 
\cite[Lemma 4.18]{Troltzsch}
\begin{equation}
\label{eq:variational_inequality}
j'(\bar{\mathbf{u}}) ( \mathbf{u} - \bar{ \mathbf{u}}) \geq 0 \quad \forall  \mathbf{u} \in \mathbf{U}_{ad}.
\end{equation}
Here, $j'(\bar{\mathbf{u}})$ denotes the Gate\^aux derivative of  $j$ at $\bar{\mathbf{u}}$. \EO{To explore this inequality, we introduce the adjoint variable $\mathsf{p}$ as the unique solution to the \emph{adjoint equation}: Find $\mathsf{p} \in H_0^1(\Omega)$ such that}
\begin{equation}\label{eq:adj_eq}
(\nabla w,\nabla \mathsf{p})_{L^2(\Omega)} + \left(\tfrac{\partial a}{\partial y}(\cdot,\mathsf{y})\mathsf{p},w\right)_{L^2(\Omega)} = (\mathsf{y} - \mathsf{y}_d,w)_{L^2(\Omega)}
\quad
\forall w \in H_0^1(\Omega),
\quad 
\mathsf{y} = \mathcal{S} \mathbf{u}.
\end{equation} 

\EO{In the following results, we explore regularity properties for $\mathsf{p}$.}

\begin{proposition}[global regularity properties of $\mathsf{p}$]
\EO{Let $\Omega$ be an open and bounded domain with Lipschitz boundary. If $a$ is such that \textnormal{\ref{A1} and \ref{A2}} hold, then there exists $\mathtt{q}>4$ if $d=2$ and $\mathtt{q}>3$ if $d=3$ such that $\mathsf{p} \in H_0^1(\Omega) \cap W_{\scaleto{0}{4.5pt}}^{1,\mathtt{q}}(\Omega)$.} 
\label{pro:regularity_p}
\end{proposition}
\begin{proof}
\EO{We proceed in three dimensions; the analysis in two dimensions is simpler. Since $\mathsf{y} \in W^{1,p}(\Omega)$, the Sobolev embedding $W^{1,p}(\Omega) \hookrightarrow L^{\iota}(\Omega)$, which holds for every $\iota<3$ \EOP{\cite[Theorem 4.12, Part I, Case C]{MR2424078}}, and assumption \ref{A2} allow us to deduce the existence of $q>d/2$ such that $\partial a/ \partial y (\cdot, \mathsf{y}) \in L^{q}(\Omega)$. On the other hand, we observe that $\mathsf{y} - \mathsf{y}_d \in L^2(\Omega)$. With these ingredients at hand, we thus invoke \cite[Theorem 4.2]{MR192177} to deduce that $\mathsf{p} \in L^{\infty}(\Omega)$.}
 \EO{Let us now define 
\begin{equation*}
\mathfrak{H}: H_0^1(\Omega) \rightarrow \mathbb{R},
\qquad
\mathfrak{H}(w) := (\mathsf{y} - \mathsf{y}_d,w)_{L^2(\Omega)} -  \left(\tfrac{\partial a}{\partial y}(\cdot,\mathsf{y})\mathsf{p},w\right)_{L^2(\Omega)}.
\end{equation*}
In what follows, we prove the existence of $\mathtt{q}>3$ such that $\mathfrak{H}$ belongs to $W^{-1,\mathtt{q}}(\Omega)$.
Let $\texttt{p}$ be the H\"older conjugate of $\texttt{q}$. Let us consider, at the moment, $\mathtt{p}$ such that $6/5 \leq \mathtt{p}<3$ ($3/2 < \mathtt{q} \leq 6$). Observe that
\[
 \| \mathsf{y} - \mathsf{y}_d \|_{W^{-1,\mathtt{q}}(\Omega)} = \sup_{v \in W_0^{1,\mathtt{p}}(\Omega) } \frac{\langle \mathsf{y} - \mathsf{y}_d,v  \rangle}{\| \nabla v\|_{L^\mathtt{p}(\Omega)}} \lesssim \| \mathsf{y} - \mathsf{y}_d \|_{L^2(\Omega)},
\]
as a consequence of the standard Sobolev embedding $W^{1,\mathtt{p}}(\Omega) \hookrightarrow L^2(\Omega)$ \EOP{\cite[Theorem 4.12, Part I, Case C]{MR2424078}.} To bound the term $\| \partial a/\partial y(\cdot,\mathsf{y})\mathsf{p} \|_{W^{-1,\mathtt{q}}(\Omega)}$, we utilize H\"older's inequality:
\[
 \left \| \tfrac{\partial a}{\partial y}(\cdot,\mathsf{y})\mathsf{p} \right \|_{W^{-1,\mathtt{q}}(\Omega)} \leq \sup_{v \in W_0^{1,\mathtt{p}}(\Omega) } 
 \left \| \tfrac{\partial a}{\partial y}(\cdot,\mathsf{y}) \right \|_{L^{\sigma}(\Omega)}
 \frac{ \| \mathsf{p} \|_{L^{\infty}(\Omega)} \| v \|_{L^{\kappa}(\Omega)}}{\| \nabla v\|_{L^\mathtt{p}(\Omega)}}, 
 \qquad \frac{1}{\sigma} + \frac{1}{\kappa} = 1.
\]
In view of the Sobolev embedding $W^{1,\mathtt{p}}(\Omega) \hookrightarrow L^{\iota}(\Omega)$, which holds for every $\iota<3$, and assumption \ref{A2}, we deduce the existence of $\mathtt{q}>3$ such that
\[
 \left \| \tfrac{\partial a}{\partial y}(\cdot,\mathsf{y})\mathsf{p} \right \|_{W^{-1,\mathtt{q}}(\Omega)} 
 \lesssim
 \left \| \tfrac{\partial a}{\partial y}(\cdot,\mathsf{y}) \right \|_{L^{q}(\Omega)}\| \mathsf{p} \|_{L^{\infty}(\Omega)}.
\]
We have thus proved the existence of $\mathtt{q}>3$ such that $\mathfrak{H} \in W^{-1,\mathtt{q}}(\Omega)$. The desired regularity result thus follows an application of \cite[Theorem 0.5]{MR1331981}.}
\end{proof}

\begin{proposition}[local regularity properties]
\EO{Let $\Omega_1 \Subset \Omega_0 \Subset \Omega$, with $\Omega_0$ smooth. If, in addition to the conditions of Proposition \ref{pro:regularity_p}, we assume that $\Omega$ is convex, $\partial a/\partial y(\cdot,y) \in L^2(\Omega)$, for every $y \in \mathbb{R}$, and $\mathsf{y}_d, \partial a/\partial y(\cdot,y) \in L^t(\Omega_0)$, for every $y \in \mathbb{R}$, where $t < \infty$ if $d = 2$ and $t < 3$ if $d=3$, then $\mathsf{p} \in W^{2,t}(\Omega_1)$.}
\label{pro:local_regularity}
\end{proposition}
\begin{proof}
\EO{Since $\mathsf{p} \in L^{\infty}(\Omega)$, we have that $\mathsf{y} - \mathsf{y}_d - \tfrac{\partial a}{\partial y}(\cdot,\mathsf{y})\mathsf{p} \in L^{2}(\Omega)$ and $\mathsf{y} - \mathsf{y}_d - \tfrac{\partial a}{\partial y}(\cdot,\mathsf{y})\mathsf{p} \in L^{t}(\Omega_0)$. The desired result thus follows from}
 \cite[Lemma 4.2]{MR3973329}:
\begin{equation}
\| \mathsf{p} \|_{W^{2,t}(\Omega_1)} \leq C_{t} \left( \left \| \mathsf{y} - \mathsf{y}_d - \tfrac{\partial a}{\partial y}(\cdot,\mathsf{y})\mathsf{p} \right \|_{L^t(\Omega_0)} 
+
 \left \| \mathsf{y} - \mathsf{y}_d - \tfrac{\partial a}{\partial y}(\cdot,\mathsf{y})\mathsf{p} \right \|_{L^2(\Omega)}\right),
\label{eq:regularity_p_W2t}
\end{equation}
where $C_t$ behaves as $Ct$, with $C>0$, as $t \uparrow \infty$. \EO{This concludes the proof.}
\end{proof}

\subsection{First order optimality conditions}
\label{sec:first_order_o_c}
We are now in position to present first order necessary optimality conditions.

\begin{theorem}[first order optimality conditions]
\EO{Let $\Omega$ be an open and bounded domain with Lipschitz boundary.  Let $a$ be such that \textnormal{\ref{A1} and \ref{A2}} hold. Then, every locally optimal control $\bar{\mathbf{u}}= \{ \bar{\mathsf{u}}_{\mathsf{z}} \}_{\mathsf{z} \in \mathcal{D}} \in \mathbf{U}_{ad}$ for problem \eqref{eq:min}--\eqref{eq:state_equation} satisfies the variational inequality}
\begin{equation}
\sum_{\mathsf{z} \in \mathcal{D}} (\bar{\mathsf{p}}(\mathsf{z}) + \alpha \bar{\mathsf{u}}_{\mathsf{z}}) (\mathsf{u}_{\mathsf{z}} -\bar{\mathsf{u}}_{\mathsf{z}}) \geq 0 \quad \forall \mathbf{u} = \{ \mathsf{u}_{\mathsf{z}} \}_{\mathsf{z} \in \mathcal{D}} \in \mathbf{U}_{ad},
\label{eq:first_order_optimality_condition}
\end{equation}
where the optimal adjoint state $\bar{\mathsf{p}}$ solves \eqref{eq:adj_eq} with $\mathsf{y}$ replaced by $\bar{\mathsf{y}} = \mathcal{S} \bar{\mathbf{u}}$.
\label{thm:first_order_op_c}
\end{theorem}
\begin{proof}
\EO{We begin the proof by observing that}
the variational inequality \eqref{eq:variational_inequality} can be rewritten as follows:
\[
0 \leq j'(\bar{\mathbf{u}}) (\mathbf{u}-\bar{\mathbf{u}})= (\mathcal{S}\bar{\mathbf{u}} - \mathsf{y}_d,\mathcal{S}'(\bar{\mathbf{u}})(\mathbf{u}-\bar{\mathbf{u}}))_{L^2(\Omega)}
+
\sum_{\mathsf{z} \in \mathcal{D}} \alpha \bar{\mathsf{u}}_{\mathsf{z}} (\mathsf{u}_{\mathsf{z}} -\bar{\mathsf{u}}_{\mathsf{z}}).
\]
Since the second term on the right hand side of the previous expression is already present in the desired variational inequality \eqref{eq:first_order_optimality_condition}, we concentrate on the first term. 

Define $\chi:= \mathcal{S}'(\bar{\mathbf{u}})(\mathbf{u}-\bar{\mathbf{u}})$ and observe that $\chi$ solves \eqref{eq:aux_adjoint} with $\mathsf{y}$ and \EO{$\mathsf{v}_{\mathsf{z}}$} replaced by $\bar{\mathsf{y}} = \mathcal{S}\bar{\mathbf{u}}$ and $\mathsf{u}_{\mathsf{z}} - \bar{\mathsf{u}}_{\mathsf{z}}$, respectively. Since there exists $\mathtt{q}>d$ such that $\bar{\mathsf{p}} \in W_{\scaleto{0}{4pt}}^{\scaleto{1}{4pt},\scaleto{\mathtt{q}}{4pt}}(\Omega)$, we are allowed to set $v = \bar{\mathsf{p}}$ in the problem that $\chi$ solves \EO{to obtain}
\begin{equation}
(\nabla \chi, \nabla \bar{\mathsf{p}} )_{L^2(\Omega)} +\left(\tfrac{\partial a}{\partial y}(\cdot,\bar{\mathsf{y}})\chi, \bar{\mathsf{p}}  \right)_{L^2(\Omega)}=
\sum_{\mathsf{z} \in \mathcal{D}} \bar{\mathsf{p}}(\mathsf{z})(\mathsf{u}_{\mathsf{z}} - \bar{\mathsf{u}}_{\mathsf{z}}).
\label{eq:aux_aux_step_first_order}
\end{equation}

Now, we would like to set $w = \chi$ in problem \eqref{eq:adj_eq} to conclude that
\begin{equation}
(\nabla \chi,\nabla \bar{\mathsf{p}})_{L^2(\Omega)} + \left(\tfrac{\partial a}{\partial y}(\cdot,\bar{\mathsf{y}})\bar{\mathsf{p}},\chi\right)_{L^2(\Omega)} = (\bar{\mathsf{y}} - \mathsf{y}_d,\chi)_{L^2(\Omega)}.
\label{eq:aux_step_first_order}
\end{equation}
Unfortunately, $\chi \notin H_0^1(\Omega)$ and thus we need to justify \eqref{eq:aux_step_first_order} with a different argument. Let $\{ \eta_n \}_{n \in \mathbb{N}}$ be \EO{a sequence in $C_0^{\infty}(\Omega)$ such that $\eta_n  \rightarrow \chi$ in $W_{\scaleto{0}{4pt}}^{\scaleto{1}{4pt},\scaleto{p}{4pt}}(\Omega)$}, as $n \uparrow \infty$, for $p<d/(d-1)$, \EO{with $p$ being arbitrarily close $d/(d-1)$.} Setting $w = \eta_n$ in \eqref{eq:adj_eq} yields
\[
(\nabla  \eta_n,\nabla \bar{\mathsf{p}})_{L^2(\Omega)} + \left(\tfrac{\partial a}{\partial y}(\cdot,\bar{\mathsf{y}})\bar{\mathsf{p}}, \eta_n \right)_{L^2(\Omega)} = (\bar{\mathsf{y}} - \mathsf{y}_d, \eta_n)_{L^2(\Omega)}.
\]
Observe that $| (\nabla (\chi -\eta_n),\nabla \bar{\mathsf{p}})_{L^2(\Omega)} 
 | \leq \| \nabla(\chi -\eta_n ) \|_{L^{p}(\Omega)} \| \nabla \bar{\mathsf{p}}\|_{L^{p'}(\Omega)} \rightarrow 0$ as $n \uparrow \infty$; $p'$ being the H\"older conjugate of $p$. Similarly, $| (\bar{\mathsf{y}} - \mathsf{y}_d, \eta_n)_{L^2(\Omega)} -  (\bar{\mathsf{y}} - \mathsf{y}_d,\chi)_{L^2(\Omega)}| \rightarrow 0$ as $n \uparrow \infty$. It thus suffices to analyze 
\begin{equation*}
\mathsf{I}_n:=
\left| \left(\tfrac{\partial a}{\partial y}(\cdot,\bar{\mathsf{y}})\bar{\mathsf{p}},\chi\right)_{L^2(\Omega)} - \left(\tfrac{\partial a}{\partial y}(\cdot,\bar{\mathsf{y}})\bar{\mathsf{p}}, \eta_n \right)_{L^2(\Omega)} \right| 
\leq \| \chi - \eta_n \|_{L^\mathfrak{q}(\Omega)} 
\EO{\| \bar{\mathsf{p}} \|_{L^{\infty}(\Omega)}}
\left\| \tfrac{\partial a}{\partial y}(\cdot,\bar{\mathsf{y}}) \right\|_{L^{\mathfrak{p}}(\Omega)},
\end{equation*}
where $\mathfrak{q} \in (1,\infty)$ is such that $\mathfrak{q} < \infty$ if $d=2$ and $\mathfrak{q} < 3$ if $d=3$ and $\mathfrak{p}$ is such that $\mathfrak{p}^{-1} + \mathfrak{q}^{-1} = 1$. \EO{The fact that $\bar{\mathsf{p}} \in L^{\infty}(\Omega)$ follows from the results of Proposition \ref{pro:regularity_p}.} 
\EO{On the other hand, we observe that, in view of \ref{A2}, there exists} $q > d/2$ such that $\tfrac{\partial a}{\partial y}(\cdot,\bar{\mathsf{y}}) \in L^{q}(\Omega)$. We thus utilize that 
\[
\| \nabla(\chi - \eta_n) \|_{L^{p}(\Omega)} \rightarrow 0 \implies
\| \chi - \eta_n \|_{L^{\mathfrak{q}}(\Omega)} \rightarrow 0,
\qquad n \uparrow \infty,
\]
to conclude that $\mathsf{I}_n \rightarrow 0$ as $n \uparrow \infty$.

Finally, we invoke \eqref{eq:aux_aux_step_first_order} and \eqref{eq:aux_step_first_order} to \EO{obtain the relation} $(\bar{\mathsf{y}} - \mathsf{y}_d,\chi)_{L^2(\Omega)} = \sum_{\mathsf{z} \in \mathcal{D}} \bar{\mathsf{p}}(\mathsf{z})(\mathsf{u}_{\mathsf{z}} - \bar{\mathsf{u}}_{\mathsf{z}})$. This concludes the proof.
\end{proof}

We \EO{conclude the} section with the following projection formula: If $\bar{\mathbf{u}} = \{ \bar{\mathsf{u}}_{\mathsf{z}} \}_{\mathsf{z} \in \mathcal{D}}  \in \mathbf{U}_{ad}$ denotes a local minimizer of problem \eqref{eq:min}--\eqref{eq:state_equation}, then, for every $\mathsf{z} \in \mathcal{D}$, we have \EO{the projection formula}
\begin{equation}\label{eq:proj_formula}
\bar{\mathsf{u}}_\mathsf{z} := \Pi_{[\mathsf{a}_{\mathsf{z}},\mathsf{b}_{\mathsf{z}}]}\left( - \alpha^{-1} \bar{\mathsf{p}} (\mathsf{z}) \right),
\end{equation}
where, for $t \in \mathbb{R}$, $\Pi_{[\mathsf{a}_{\mathsf{z}},\mathsf{b}_{\mathsf{z}}]}(t):= \max \{ \mathsf{a}_{\mathsf{z}}, \min \{ \mathsf{b}_{\mathsf{z}}, t \} \}$.


\subsection{Second order optimality conditions}
\label{sec:second_order_o_c}
In this section, \EO{we provide} necessary and sufficient second order optimality conditions for \eqref{eq:min}--\eqref{eq:state_equation}.

\subsubsection{Second order differentiability \EO{of $j$}}
\EO{Before providing second order optimality conditions, we analyze second order differentiability properties for the reduced cost functional} $j$.

\begin{proposition}[second order differentiability of the reduced cost functional $j$]
\label{thm:second_order_diff_j}
\EO{Let $\Omega$ be an open and bounded domain with Lipschitz boundary. If $a$ satisfies \textnormal{\ref{A1}, \ref{A2}, and \ref{A3}}, then the reduced cost functional $j: \mathbb{R}^{\ell} \rightarrow \mathbb{R}$ is of class $C^2$. In addition, for $\mathbf{u}, \mathbf{v}, \mathbf{w} \in \mathbb{R}^{\ell}$, we have}
\begin{equation}\label{eq:second_order_j}
j''(\mathbf{u})(\mathbf{v},\mathbf{w}) 
= 
\int_{\Omega} \phi_{\mathbf{v}} \phi_{\mathbf{w}} \mathrm{d}x 
+
\sum_{\mathsf{z} \in \mathcal{D}} \alpha \mathsf{v}_{\mathsf{z}}\mathsf{w}_{\mathsf{z}} 
- 
\int_{\Omega} \mathsf{p} \frac{\partial^2 a}{\partial y^2}(x,\mathsf{y}) \phi_{\mathbf{v}} \phi_{\mathbf{w}} \mathrm{d}x.
\end{equation}
\EO{Here, $\mathsf{y} = \mathcal{S}\mathbf{u}$, $\phi_{\mathbf{v}} = \mathcal{S}'(\mathbf{u})\mathbf{v}$, and $\phi_{\mathbf{v}} = \mathcal{S}'(\mathbf{u})\mathbf{w}$.}
\end{proposition}
\begin{proof}
Since Theorem \ref{thm:first_order_diff} guarantees that $\mathcal{S}: \EO{\mathbb{R}^{\ell}} \rightarrow W_0^{1,p}(\Omega)$ is twice continuously Fr\'echet differentiable, it is immediate that $j$ is of class $C^2$. The identity \eqref{eq:second_order_j} follows from computations similar to those in \EO{\cite[Section 4.10, page 241]{Troltzsch}} upon noticing that
\[
- \int_{\Omega} \mathsf{p} \frac{\partial^2 a}{\partial y^2}(x,\mathsf{y}) \phi_{\mathbf{v}} \phi_{\mathbf{w}} \mathrm{d}x = \int_{\Omega} (\mathsf{y} - \mathsf{y}_d) \varphi \mathrm{d}x.
\]
\EO{Here, $\varphi$ denotes the solution to \eqref{eq:aux_second_adjoint_new}.}
This identity follows from the density argument elaborated in the proof of Theorem \ref{thm:first_order_op_c} combined with the fact that $a$ satisfies \ref{A1}--\ref{A3}. \EO{We observe that, in view of \ref{A2}, the results of Proposition \ref{pro:regularity_p} allow us to conclude that $\mathsf{p} \in L^{\infty}(\Omega)$ while assumption \ref{A3} reveals that $\partial^2 a/\partial y^2(\cdot,\mathsf{y}) \phi_{\mathbf{v}} \phi_{\mathbf{w}} \in L^1(\Omega)$. As a result, $\mathsf{p}\partial^2 a/\partial y^2(\cdot,\mathsf{y}) \phi_{\mathbf{v}} \phi_{\mathbf{w}} \in L^1(\Omega)$} 
\end{proof}

\subsubsection{Second order necessary optimality conditions}
 
\EO{Before presenting necessary and sufficient second order optimality conditions, we introduce a few basic ingredients. Let us define}
\begin{equation}\label{def:deriv_j}
\Psi:= \{ \psi_{\mathsf{z}} \}_{\mathsf{z}\in\mathcal{D}} \in\mathbb{R}^{\ell},
\qquad
\psi_{\mathsf{z}}:=\bar{\mathsf{p}}(\mathsf{z})+\alpha \bar{\mathsf{u}}_{\mathsf{z}}.
\end{equation}
\EO{Let us also introduce the cone of critical directions at $\bar{\mathbf{u}} \in \mathbb{U}_{ad}$:}
\begin{equation}
\mathbf{C}_{\bar{\mathbf{u}}}:=\left \{\mathbf{v}=\{ \mathsf{v}_{\mathsf{z}} \}_{\mathsf{z}\in\mathcal{D}} \in \mathbb{R}^{\ell}
 ~\text{satisfying}~\eqref{eq:sign_cond} 
~\text{and},~\text{for}~ \mathsf{z}\in\mathcal{D},~
\mathsf{v}_{\mathsf{z}} = 0 \text{ if } \psi_{\mathsf{z}}\neq 0
\right \}.
\label{eq:cone_critical}
\end{equation}
The \EO{aforementioned} condition \eqref{eq:sign_cond} reads as follows: For every $\mathsf{z}\in\mathcal{D}$, we have
\begin{equation}\label{eq:sign_cond}
\mathsf{v}_{\mathsf{z}}
\geq 0 ~\text{ if }~ \bar{\mathsf{u}}_{\mathsf{z}}=\mathsf{a}_{\mathsf{z}},
\qquad
\mathsf{v}_{\mathsf{z}}
\leq 0 ~\text{ if }~ \bar{\mathsf{u}}_{\mathsf{z}}=\mathsf{b}_{\mathsf{z}}.
\end{equation}

\EO{As stated in \cite[Section 3.3]{MR3311948}, the following result follows from the standard Karush--Kuhn--Tucker theory of mathematical optimization in finite-dimensional spaces; see, for instance, \cite[Theorem 3.8]{MR3311948} and \cite[Section 6.3]{MR2012832}.}
%

\begin{theorem}[\EO{second order necessary and sufficient optimality conditions}] \EO{If $\bar{\mathbf{u}}\in\mathbf{U}_{ad}$ denotes a local minimum for problem \eqref{eq:min}--\eqref{eq:state_equation}, then
$
j''(\bar{\mathbf{u}})\mathbf{v}^2 \geq 0
$
for all $\mathbf{v}\in \mathbf{C}_{\bar{\mathbf{u}}}$. Conversely, if $\bar{\mathbf{u}}\in\mathbf{U}_{ad}$ satisfies the variational inequality \eqref{eq:first_order_optimality_condition} and the second order sufficient optimality condition
\begin{equation}
\label{eq:second_order_sufficient_condition}
 j''(\bar{\mathbf{u}})\mathbf{v}^2 > 0 \quad \forall \mathbf{v} \in \mathbf{C}_{\bar{\mathbf{u}}} \setminus \{ \mathbf{0} \},
\end{equation}
then there exists $\mu>0$ and $\sigma >0$ such that
\begin{equation*}
j(\mathbf{u})
\geq 
j(\bar{\mathbf{u}})
+
\frac{\mu}{2} \|   \bar{\mathbf{u}} - \mathbf{u} \|^2_{\mathbb{R}^{\ell}}
\quad
\forall \mathbf{u} \in \mathbf{U}_{ad}:  \|   \bar{\mathbf{u}} - \mathbf{u} \|_{\mathbb{R}^{\ell}} \leq \sigma.
\end{equation*}
In particular, $\bar{\mathbf{u}}$ is a strict local solution to \eqref{eq:min}--\eqref{eq:state_equation}.}
\label{thm:second_order_necessary}
\end{theorem}

To present the following result, we define, for $\tau >0$, \EO{the cone}
\begin{equation}
\mathbf{C}_{\bar{\mathbf{u}}}^{\tau}:= \left \{\mathbf{v}=\{ \mathsf{v}_{\mathsf{z}} \}_{\mathsf{z}\in\mathcal{D}} \in \mathbb{R}^{\ell}
 ~\text{satisfying}~\eqref{eq:sign_cond}  
 ~\text{and}, ~\text{for}~ \mathsf{z}\in\mathcal{D},
\mathsf{v}_{\mathsf{z}} = 0 \text{ if } |\psi_{\mathsf{z}} | > \tau
 \right \}.
\label{eq:cone_critical_tau}
\end{equation}

\EO{The following result is immediate in view of our finite-dimensional setting.}

\begin{theorem}[\EO{equivalent optimality conditions}]
\label{thm:second_order_equivalent}
Let $\bar{\mathbf{u}} \in\mathbf{U}_{ad}$, $\bar{\mathsf{y}} \in W_0^{1,p}(\Omega)$, and $\bar{\mathsf{p}} \in H_0^1(\Omega)$ satisfy the first order optimality conditions \eqref{eq:state_equation}, \eqref{eq:adj_eq}, and \eqref{eq:first_order_optimality_condition}. \EO{Then, \eqref{eq:second_order_sufficient_condition} is equivalent to} 
\begin{equation}
\label{eq:second_order_condition_equivalent}
\exists \kappa, \tau>0:
\quad
j''(\bar{\mathbf{u}})\mathbf{v}^2 \geq \kappa \| \mathbf{v} \|^2_{\mathbb{R}^{\ell}} \quad \forall \mathbf{v} \in \mathbf{C}_{\bar{\mathbf{u}}}^{\tau}.
\end{equation}
\end{theorem}


\section{Finite element approximation of the optimal control problem}
\label{sec:fem_ocp}
In this section, we propose and analyze a finite element discretization scheme for the optimal control problem \eqref{eq:min}--\eqref{eq:state_equation}. In particular, we analyze convergence properties and derive error estimates. 

We begin our studies by providing convergence results related to a finite element discretization of the state equation \eqref{eq:state_equation}.

\subsection{The discrete state equation: convergence properties}
\label{sec:discrete_PDE_convergence}

We introduce the following finite element approximation of problem \eqref{eq:state_equation}: Find $\mathpzc{y}_h \in \mathbb{V}_h$ such that
\begin{equation}
\label{eq:state_equation_discrete}
 \int_{\Omega} \nabla \mathpzc{y}_h \cdot \nabla v_h \mathrm{d}x + \int_{\Omega} a(x,\mathpzc{y}_h) v_h \mathrm{d}x = \sum_{\mathsf{z} \in \mathcal{D} } \mathsf{u}_{\mathsf{z}} v_{h} (\mathsf{z}) 
 \quad \forall v_h \in \mathbb{V}_h.
\end{equation}
\EO{We recall that $\mathbb{V}_{h}$ is defined in \eqref{eq:piecewise_linear_set}. The existence of a discrete solution $\mathpzc{y}_h \in \mathbb{V}_{h}$, for a fixed mesh $\T_h$, follows from Brouwer's fixed point theorem while uniqueness follows from the monotonicity of $a$.}

We present the following convergence result.

\begin{theorem}[convergence properties]
Let $\Omega$ be an open, bounded, and convex polytopal domain. \EO{Let $a$ be such that \textnormal{\ref{A1}, \ref{A2}}, and \eqref{eq:a_Lipschitz} hold.} Let $\mathsf{y} \in W_{\scaleto{0}{4pt}}^{\scaleto{1}{4pt},\scaleto{p}{4pt}}(\Omega)$, for $p<d/(d-1)$, be the solution to the \EO{state equation} \eqref{eq:state_equation} and let $\mathsf{y}_h \in \mathbb{V}_h$ be the solution to the discrete equation
\begin{equation}\label{eq:weak_semilinear_discrete_convergence}
\int_{\Omega} \nabla \mathsf{y}_h \cdot \nabla v_{h} \mathrm{d}x + \int_{\Omega} a(x,\mathsf{y}_h) v_{h} \mathrm{d}x = \sum_{\mathsf{z} \in \mathcal{D}} \mathsf{u}_{\mathsf{z},h} v_h(\mathsf{z})
\quad \forall v_h \in \mathbb{V}_h,
\end{equation}
where \EO{$\mathbf{u}_h =  \{ \mathsf{u}_{\mathsf{z},h} \}_{\mathsf{z} \in \mathcal{D}} \in \mathbf{U}_{ad}$.}
 If $\mathbf{u}_h \rightarrow \mathbf{u}$ in $\mathbb{R}^{\ell}$, then $\mathsf{y}_h \rightarrow \mathsf{y}$ in $L^2(\Omega)$ as $h \rightarrow 0$.
\label{thm:convergence_state_equation}
\end{theorem}
\begin{proof}
\EO{We begin the proof with a simple application of the triangle inequality and write
\[
\| \mathsf{y} -  \mathsf{y}_h \|_{L^{2}(\Omega)} 
\leq
\| \mathsf{y} - \mathfrak{y} \|_{L^{2}(\Omega)} 
+
\| \mathfrak{y} - \mathsf{y}_h \|_{L^{2}(\Omega)},
\]
where $\mathfrak{y}$ denotes the solution to: Find $\mathfrak{y} \in W_0^{1,p}(\Omega)$ such that
\begin{equation}
\int_{\Omega} \nabla \mathfrak{y} \cdot \nabla v \mathrm{d}x 
+
\int_{\Omega} a(x,\mathfrak{y})  v \mathrm{d}x = \sum_{\mathsf{z} \in \mathcal{D} } \mathsf{u}_{\mathsf{z},h}  v(\mathsf{z})
\quad
\forall v \in W_0^{1,p'}(\Omega),
\quad
p<d/(d-1).
\label{eq:fraky}
\end{equation}
To} \EO{control $\| \mathsf{y} - \mathfrak{y} \|_{L^{2}(\Omega)}$, we write $a(x,\mathsf{y}) - a(x,\mathfrak{y})$ as $c_0 (\mathsf{y}-\mathfrak{y})$, where $c_0 := \int_0^1 \partial a / \partial y (x, \zeta )\mathrm{d}\theta$ with $\zeta:= \mathsf{y} + \theta(\mathfrak{y}-\mathsf{y})$, and observe that $\mathsf{y} - \mathfrak{y}$ solves the problem
\begin{equation*}
\int_{\Omega} \nabla (\mathsf{y}-\mathfrak{y}) \cdot \nabla v \mathrm{d}x 
+
\int_{\Omega} c_0 (\mathsf{y} - \mathfrak{y}) v \mathrm{d}x = \sum_{\mathsf{z} \in \mathcal{D} } [ \mathsf{u}_{\mathsf{z}}-\mathsf{u}_{\mathsf{z},h}] v(\mathsf{z})
\quad
\forall v \in W_0^{1,p'}(\Omega),
\quad
p<d/(d-1).
\label{eq:y_fraky}
\end{equation*}
Since both $\mathsf{y}$ and $\mathfrak{y}$ belong to $W_0^{1,p}(\Omega)$ for every $p<d/(d-1)$, we deduce that $\zeta = \mathsf{y} + \theta(\mathfrak{y}-\mathsf{y}) \in L^r(\Omega)$ for every $r < \infty$ if $d = 2$ and for every $r<3$ if $d=3$ \EOP{\cite[Theorem 4.12, Part I, Case C]{MR2424078}}. Consequently, \ref{A2} guarantees that $\partial a / \partial y ( \cdot, \zeta )\in L^q(\Omega)$ for some $q>d/2$. We are thus in position to invoke the stability estimate of \cite[Theorem 9.1]{MR192177} to arrive at
\begin{equation}
\| \nabla (\mathsf{y}-\mathfrak{y})\|_{L^{p}(\Omega)} \lesssim \| \mathbf{u} - \mathbf{u}_h \|_{\mathbb{R}^{\ell}} 
\rightarrow 0,
\qquad
h \rightarrow 0.
\label{eq:first_estimate}
\end{equation}
We now} \EO{control $\| \mathfrak{y} - \mathsf{y}_h \|_{L^{2}(\Omega)}$. Since $\mathfrak{y} \in W_0^{1,p}(\Omega)$ solves \eqref{eq:fraky} and $\mathsf{y}_h \in \mathbb{V}_h$ solves \eqref{eq:weak_semilinear_discrete_convergence}, an immediate 
application of the error estimate of Theorem \ref{thm:eq:error_estimate_L2_semilinear} yields the existence of $h_{\triangle}>0$ such that 
\begin{equation}
 \| \mathfrak{y} - \mathsf{y}_h \|_{L^{2}(\Omega)} \lesssim h^{2-d/2} 
\qquad \forall h \leq h_{\triangle}.
\label{eq:second_estimate}
\end{equation}
The hidden constant is independent of $h$, but depends on $\| \mathbf{u}_h \|_{\mathbb{R}^{\ell}}$, $\| \chi \|_{L^2(\Omega)}$, $\| a(\cdot,0) \|_{L^1(\Omega)}$, $\| \phi_0 \|_{L^1(\Omega)}$, and $r$ as in assumption \ref{A1}. We observe that, since $\{ \mathbf{u}_h \}_{h>0} \subset \mathbf{U}_{ad}$ is convergent to $\mathbf{u} \in \mathbf{U}_{ad}$ as $h \rightarrow 0$, it is uniformly bounded with respect to discretization. A collection of the convergence property \eqref{eq:first_estimate} and the error estimate \eqref{eq:second_estimate} allows us to conclude.}
\end{proof}

\subsection{The discrete adjoint equation: \EO{local error estimates in maximum norm}}
\label{sec:discrete_adjoint_convergence}

Let us introduce the following finite element approximation of problem \eqref{eq:adj_eq}: Find $\mathpzc{p}_h \in \mathbb{V}_h$ such that
\begin{equation}
\label{eq:adjoint_discrete}
 \int_{\Omega}  \nabla v_h \cdot \nabla \mathpzc{p}_h \mathrm{d}x + \int_{\Omega} \tfrac{\partial a}{\partial y}(x,\mathsf{y}) \mathpzc{p}_h v_h \mathrm{d}x = \int_{\Omega} (\mathsf{y} - \mathsf{y}_d) v_h \mathrm{d}x \quad \forall v_h \in \mathbb{V}_h.
\end{equation}

\EO{In the following result, we present a \emph{two-dimensional} local error estimate in maximum norm.}

\begin{theorem}[\EO{an error estimate in maximum norm: $d=2$}]\label{thm:error_estimate_adj}
\EOP{Let $\Omega \subset \mathbb{R}^2$ be an open, bounded, and convex polygonal domain. Let $a$ be such that \textnormal{\ref{A1}} and \textnormal{\ref{A2}} hold. Let $\Omega_1 \Subset \Omega_0 \Subset \Omega$ with $\Omega_0$ smooth. If, in addition, $\partial a/\partial y(\cdot,y) \in L^{\mathfrak{q}}(\Omega)$, for every $y \in \mathbb{R}$, where $\mathfrak{q}>2$, and $\mathsf{y}_d, \partial a/\partial y(\cdot,y) \in L^t(\Omega_0)$, for every $y \in \mathbb{R}$, where $t < \infty$, then there exists $h_{\star}>0$ such that 
\begin{equation}
\| \mathsf{p} - \mathpzc{p}_h \|_{L^{\infty}(\Omega_1)} \lesssim h^2|\log h|^2 
\label{eq:local_error_estimate}
\end{equation}
for every $h \leq h_{\star}$. Here, the hidden constant is independent of $h$.}
\end{theorem}
\begin{proof}
\EO{We begin the proof with a simple application of the triangle inequality and write}
\[
\| \mathsf{p}  -\mathpzc{p}_h\|_{L^{\infty}(\Omega_1)} 
\leq 
\| \mathsf{p}  - q_h \|_{L^{\infty}(\Omega_1)} 
+
\| q_h  -\mathpzc{p}_h \|_{L^{\infty}(\Omega_1)},
\]
where $q_h \in \mathbb{V}_h$ solves \eqref{eq:adjoint_discrete} \EO{but with} $\partial a / \partial y(\cdot,\mathsf{y})\mathpzc{p}_h$ replaced by $\partial a / \partial y(\cdot,\mathsf{y})\mathsf{p}$. A key ingredient in favor of the definition of $q_h$ is that $(\nabla(\mathsf{p} - q_h), \nabla v_h)_{L^2(\Omega)} = 0$ for every $v_h \in \mathbb{V}_h$. Let $\Lambda_1$ be a smooth domain such that $\Omega_1 \Subset \Lambda_1 \Subset \Omega_0$. We thus invoke \cite[Corollary 5.1]{MR431753} and \cite[Proposition 4.3]{MR3973329} to conclude the existence of $h_0 \in (0,1)$ and $C>0$ such that, for \EO{any} $v_h \in \mathbb{V}_h$,
\[
\| \mathsf{p}  - q_h \|_{L^{\infty}(\Omega_1)} \lesssim |\log \mathfrak{l} h| \| \mathsf{p}  - v_h \|_{L^{\infty}(\Lambda_1)}
+ \mathfrak{l}^{-\frac{d}{2}}  \| \mathsf{p}  - q_h \|_{L^{2}(\Lambda_1)}
\quad
\forall h \leq h_0.
\]
Here, $\mathfrak{l}$ is such that $\mathrm{dist}(\Omega_1,\partial \Lambda_1) \geq \mathfrak{l}$, $\mathrm{dist}(\Lambda_1,\partial \Omega) \geq \mathfrak{l}$, and $C h \leq \mathfrak{l}$. \EO{In view of the assumptions on $a$ and the convexity of $\Omega$, the results of Proposition \ref{pro:local_regularity} guarantee that $\mathsf{p} \in H^2(\Omega) \cap W^{2,t}(\Lambda_1)$, where $t$ is as in the statement of the theorem. The $H^2(\Omega)$-regularity of $\mathsf{p}$ follows from the fact that $\partial a/\partial y(\cdot,y) \in L^{2}(\Omega)$, for every $y \in \mathbb{R}$, and that $\mathsf{p} \in L^{\infty}(\Omega)$; see Proposition \ref{pro:regularity_p}.} We can thus obtain 
\begin{equation}
\begin{aligned}
\| \mathsf{p}  - q_h \|_{L^{\infty}(\Omega_1)} & \leq C_1|\log h| h^{2-\frac{d}{t}} \EO{\| \nabla^2 \mathsf{p} \|_{L^t(\Lambda_1)}}
+ C_2 h^2 |\mathsf{p}|_{H^2(\Omega)}
\\
& \leq
C_t|\log h| h^{2-\frac{d}{t}} \mathfrak{C}
+ C_2 h^2 |\mathsf{p}|_{H^2(\Omega)},
\end{aligned}
\label{eq:local_error_estimate_aux}
\end{equation}
where $C_1, C_2$ and $C_t$ are positive constants that are independent of $h$ and $\mathsf{p}$ and $\mathfrak{C} = \mathfrak{C}(\mathsf{y},\mathsf{y}_d,a,\mathsf{p})$ collects the terms appearing in the right hand side of estimate \eqref{eq:regularity_p_W2t}. 
\EO{Inspired by \cite[page 3]{MR637283}, we set $t = |\log h|$, for $h$ sufficiently small, and use that $C_t$ behaves as $Ct$ as $t\uparrow\infty$, to obtain the error estimate $\| \mathsf{p}  - q_h \|_{L^{\infty}(\Omega_1)} \lesssim h^{\scaleto{2}{4pt}}|\log h|^{\scaleto{2}{4pt}}$. 
}

We now control $\| q_h  - \mathpzc{p}_h \|_{L^{\infty}(\Omega_1)}$. To accomplish this task, we first notice that 
\begin{equation}
q_h  - \mathpzc{p}_h \in \mathbb{V}_h:
\quad
\int_{\Omega} \nabla(q_h  - \mathpzc{p}_h) \cdot \nabla v_h \mathrm{d}x = \int_{\Omega} \tfrac{\partial a}{\partial y}(x,\mathsf{y}) ( \mathpzc{p}_h - \mathsf{p} )v_h \mathrm{d}x
\quad
\forall v_h \in \mathbb{V}_h.
\label{eq:qh-frakph}
\end{equation}
Let $\mathfrak{p}$ be the solution to the associated continuous problem, i.e.,
\begin{equation}
\mathfrak{p} \in H_0^1(\Omega):
\quad
\int_{\Omega} \nabla \mathfrak{p} \cdot \nabla v \mathrm{d}x = \int_{\Omega} \tfrac{\partial a}{\partial y}(x,\mathsf{y}) ( \mathpzc{p}_h - \mathsf{p} ) v \mathrm{d}x 
\quad
\forall v \in H_0^1(\Omega).
\label{eq:frakp}
\end{equation}
\EO{In what follows, we obtain a stability bound for $\mathfrak{p}$ in $L^{\infty}(\Omega)$ on the basis of a basic Sobolev embedding and the stability bound in \cite[Theorem 0.5]{MR1331981}. In fact, we have
\begin{equation}
\begin{aligned}
\| \mathfrak{p} \|_{L^{\infty}(\Omega)} & \lesssim  \| \nabla \mathfrak{p} \|_{L^{r}(\Omega)} \lesssim \left \| \tfrac{\partial a}{\partial y}(\cdot,\mathsf{y}) \right \|_{L^{\mathfrak{q}}(\Omega)} \| \mathpzc{p}_h - \mathsf{p} \|_{L^2(\Omega)},
\end{aligned}
\end{equation}
where $r > 2$ and $\mathfrak{q} > 2$. We thus control} $\| q_h  - \mathpzc{p}_h \|_{L^{\infty}(\Omega_1)}$ as follows: $\| q_h  - \mathpzc{p}_h \|_{L^{\infty}(\Omega_1)} \leq \| (q_h  - \mathpzc{p}_h) - \mathfrak{p}  \|_{L^{\infty}(\Omega_1)} +  \| \mathfrak{p}  \|_{L^{\infty}(\Omega_1)}$. \EO{Since $q_h  - \mathpzc{p}_h$ solves \eqref{eq:qh-frakph} and $\mathfrak{p}$ solves \eqref{eq:frakp}, i.e, $q_h  - \mathpzc{p}_h$ corresponds to a finite element approximation of $\mathfrak{p}$, we obtain} 
\[
\| (q_h  - \mathpzc{p}_h) - \mathfrak{p}  \|_{L^{\infty}(\Omega_1)} \lesssim 
h^{1-\frac{d}{r}} \| \nabla \mathfrak{p} \|_{L^r(\Omega)}
\lesssim 
\| \nabla \mathfrak{p} \|_{L^r(\Omega)}.
\]
\EO{Consequently, we obtain
$
\| q_h  - \mathpzc{p}_h \|_{L^{\infty}(\Omega_1)} \lesssim \|  \mathpzc{p}_h - \mathsf{p} \|_{L^2(\Omega)} \lesssim h^2 | \mathsf{p}  |_{H^2(\Omega)}.
$
This estimate combined with the error bound $\| \mathsf{p} - q_h \|_{L^{\infty}(\Omega_1)} \lesssim h^{\scaleto{2}{4pt}}|\log h|^{\scaleto{2}{4pt}}$ yield \eqref{eq:local_error_estimate}.} 
\end{proof}

\subsection{\EO{The discrete state equation: an error estimate in $L^1(\Omega)$}}
\label{sec:error_estimate_L1}
In this section, we follow the ideas developed in \cite[Theorem 4.1]{MR3116646} and \cite[Lemma 4.2]{MR3225501} and derive an error estimate in $L^1(\Omega)$ for \EO{the finite element approximation \eqref{eq:state_equation_discrete}} of the \emph{semilinear state equation} \eqref{eq:state_equation}. Notice that, since $\mathcal{D} \subset \Omega$ and $\mathcal{D}$ is finite, $\mathrm{dist}(\mathcal{D},\partial \Omega)>0$ so that we can conclude the existence of smooth subdomains $\Omega_0$ and $\Omega_1$ such that $\mathcal{D} \subset \Omega_1 \Subset \Omega_0 \Subset \Omega$.

\begin{theorem}[\EO{an error estimate in $L^1(\Omega)$: $d=2$}]
\EOP{Let $\Omega \subset \mathbb{R}^2$ be an open, bounded, and convex polygonal domain. Let $a$ be such that \textnormal{\ref{A1}} holds. Assume, in addition, that $a(\cdot,0) \in L^2(\Omega)$ 
and that \eqref{eq:a_Lipschitz} holds with $\psi \in L^s(\Omega)$, where $s>2$. Let $\Omega_0$ and $\Omega_1$ be smooth subdomains such that $\mathcal{D} \subset \Omega_1 \Subset \Omega_0 \Subset \Omega$. If $\psi \in L^t(\Omega_0)$, for every $t < \infty$, 
then, there exists $h_{\bowtie}>0$ such that
\begin{equation}
\| y - \mathpzc{y}_h \|_{L^{1}(\Omega)} \lesssim h^2|\log h|^2
\label{eq:L1_error_estimate}
\end{equation}
for every $h \leq h_{\bowtie}$. Here, the hidden constant is independent of $h$.}
\label{thm:L1_error_estimate}
\end{theorem}
\begin{proof}
Define $\chi$, $\mathfrak{z}$, and $\mathfrak{z}_h$ as in the proof of Theorem \ref{thm:eq:error_estimate_L2_semilinear}. Set $\mathfrak{f} = \mathrm{sgn}(\mathsf{y} - \mathpzc{y}_h) \in L^{\infty}(\Omega)$ as a forcing term in \EO{problem \eqref{eq:problem_z}. A density argument allows us to set $v = y - \mathpzc{y}_h$ in problem \eqref{eq:problem_z} and obtain}
\begin{equation*}
\| y - \mathpzc{y}_h \|_{L^1(\Omega)}  =
(\nabla (y-\mathpzc{y}_h), \nabla (\mathfrak{z}-\mathfrak{z}_h))_{L^2(\Omega)} 
+
(a(\cdot,y)-a(\cdot,\mathpzc{y}_h), \mathfrak{z}-\mathfrak{z}_h)_{L^2(\Omega)}.
\end{equation*}
\EO{We now exploit the fact that $y$ solves \eqref{eq:state_equation} to arrive at}
\begin{equation}
\label{eq:aux_L1}
\| y - \mathpzc{y}_h \|_{L^1(\Omega)} = \sum_{\mathsf{z}} \mathsf{u}_{\mathsf{z}} (\mathfrak{z}(\mathsf{z})-\mathfrak{z}_h(\mathsf{z}))
+ (\chi (\mathfrak{z}-\mathfrak{z}_h), \mathpzc{y}_h)_{L^2(\Omega)} 
- 
\EO{(a(\cdot,\mathpzc{y}_h),\mathfrak{z}-\mathfrak{z}_h)_{L^2(\Omega)} 
=: 
\textrm{I} + \textrm{II} + \textrm{III}. 
}
\end{equation}

\EO{We first estimate $\textrm{II} + \textrm{III}$ upon exploiting the $H^2(\Omega)$-regularity of $\mathfrak{z}$: 
\begin{equation*}
| (\chi (\mathfrak{z}-\mathfrak{z}_h), \mathpzc{y}_h)_{L^2(\Omega)} - (a(\cdot,\mathpzc{y}_h),\mathfrak{z}-\mathfrak{z}_h)| 
\leq
 \left( \| \chi \mathpzc{y}_h\|_{L^2(\Omega)} + \| a(\cdot,\mathpzc{y}_h)\|_{L^2(\Omega)} \right)
 \| \mathfrak{z}-\mathfrak{z}_h  \|_{L^2(\Omega)}
 \lesssim h^2 |  \mathfrak{z} |_{H^2(\Omega)}.
\end{equation*}

The following comments are now in order. First, since $\Omega$ is convex and the underlying refinement is quasi-uniform, we invoke \cite[Theorem 8.5.3]{MR2373954} 
to obtain the existence of $h_{\triangle}>0$ such that $y_h \in W_{\scaleto{0}{4pt}}^{\scaleto{1}{4pt},\scaleto{p}{4pt}}(\Omega)$ uniformly with respect to discretization for $h \leq h_{\triangle}$.
With this result at hand, we now utilize that $\chi \in L^s(\Omega)$ with $s>2$ to immediately deduce that $\chi y_h \in L^2(\Omega)$ in two dimensions. 
On the other hand, in view of the fact that $a(\cdot,0) \in L^2(\Omega)$, it follows that $\| a(\cdot,y_h)\|_{L^2(\Omega)}$ is uniformly bounded with respect to discretization.

It thus suffices to estimate the term $\textrm{I}$ in \eqref{eq:aux_L1}. Observe that $\mathfrak{z}$ is such that}
\[
(\nabla v, \nabla \mathfrak{z})_{L^2(\Omega)}  = (\mathrm{sgn}(\mathsf{y} - \mathpzc{y}_h) - \chi \mathfrak{z},v)_{L^2(\Omega)} \quad \forall v \in H_0^1(\Omega).
\]
We can thus invoke \cite[Lemma 4.2]{MR3973329} to obtain
\begin{equation*}
\| \mathfrak{z} \|_{W^{2,t}(\Omega_1)} \leq C_{t} \left( \left \| \mathrm{sgn}(\mathsf{y} - \mathpzc{y}_h) - \chi \mathfrak{z} \right \|_{L^t(\Omega_0)}
+
\left \|\mathrm{sgn}(\mathsf{y} - \mathpzc{y}_h) - \chi \mathfrak{z} \right \|_{L^2(\Omega)}\right),
\label{eq:regularity_zeta_W2t}
\end{equation*}
where $t$ is as in the statement of the theorem. A local argument as the one developed in the proof of Theorem \ref{thm:error_estimate_adj} yields \EO{the bound
\[
\left| 
\mathrm{I}
\right|
\lesssim 
\| \mathbf{u} \|_{\mathbb{R}^{\ell}} \| \mathfrak{z} - \mathfrak{z}_h \|_{L^{\infty}(\Omega_1)} 
\lesssim 
h^2 |\log h|^2 \| \mathbf{u} \|_{\mathbb{R}^{\ell}}.
\]
This concludes the proof.}
\end{proof}

\subsection{The discrete optimal control problem}
\label{sec:discrete_optimal_control_problem}

We propose the following finite element discretization of the optimal control problem \eqref{eq:min}--\eqref{eq:state_equation}: Find
\begin{equation}
\label{eq:min_discrete}
\min \{ J (\mathsf{y}_h,\mathbf{u}_h): (\mathsf{y}_h,\mathbf{u}_h) \in \mathbb{V}_h \times \mathbf{U}_{ad} \}
\end{equation}
subject to the discrete state equation: Find $\mathsf{y}_h \in \mathbb{V}_h$ such that
\begin{equation}
\label{eq:state_equation_discret}
 \int_{\Omega} \nabla \mathsf{y}_h \cdot \nabla v_h \mathrm{d}x + \int_{\Omega} a(x,\mathsf{y}_h) v_h \mathrm{d}x = \sum_{\mathsf{z} \in \mathcal{D} } \mathsf{u}_{\mathsf{z},h} v_h(\mathsf{z}) \qquad \forall v_h \in \mathbb{V}_h.
\end{equation}

The existence of at least one solution 
follows from the arguments developed in the proof of Theorem \ref{thm:existence_optimal_pair}. Let us introduce the discrete map $\mathcal{S}_h: \mathbf{U}_{ad} \ni \mathbf{u}_h \mapsto \mathsf{y}_h \in \mathbb{V}_h$, where $\mathsf{y}_h$ denotes the solution to \eqref{eq:state_equation_discret}, and define the reduced cost functional $j_h : \mathbf{U}_{ad} \ni \mathbf{u}_h \mapsto J(\mathcal{S}_h \mathbf{u}_h,\mathbf{u}_h) \in \mathbb{R}$. With these ingredients at hand, we formulate first order optimality conditions: every discrete locally optimal control $\bar{\mathbf{u}}_h \in \mathbf{U}_{ad}$ satisfies
\begin{equation}
\label{eq:first_order_optimality_condition_discrete}
j_h'(\bar{\mathbf{u}}_h)(\mathbf{u}_h - \bar{\mathbf{u}}_h) \geq  0 
\quad
\forall \mathbf{u}_h = \{ u_{\mathsf{z},h} \}_{\mathsf{z} \in \mathcal{D}} \in \mathbf{U}_{ad}.
\end{equation}
This variational inequality leads to the following projection formula: If $\bar{\mathbf{u}}_h$ denotes a local minimizer of the discrete optimal control problem, then, for every $\mathsf{z} \in \mathcal{D}$, \EO{we have the projection formula}
\begin{equation}\label{eq:proj_formula_discrete}
\bar{\mathsf{u}}_{\mathsf{z},h} := \Pi_{[\mathsf{a}_{\mathsf{z}},\mathsf{b}_{\mathsf{z}}]}\left( - \alpha^{-1} \bar{\mathsf{p}}_h (\mathsf{z}) \right),
\end{equation}
where, for $t \in \mathbb{R}$, $\Pi_{[\mathsf{a}_{\mathsf{z}},\mathsf{b}_{\mathsf{z}}]}(t):= \max \{ \mathsf{a}_{\mathsf{z}}, \min \{ \mathsf{b}_{\mathsf{z}}, t \} \}$. Here, $\bar{\mathsf{p}}_h$ denotes the solution to the following discrete adjoint problem: Find $\bar{\mathsf{p}}_h \in \mathbb{V}_h$ such that
\begin{equation}
\label{eq:adjoint_discrete_control}
 \int_{\Omega} \nabla \bar{\mathsf{p}}_h \cdot \nabla v_h \mathrm{d}x + \int_{\Omega} \tfrac{\partial a}{\partial y}(x,\bar{\mathsf{y}}_h) \bar{\mathsf{p}}_h v_h \mathrm{d}x = \int_{\Omega} (\bar{\mathsf{y}}_h - \mathsf{y}_d) v_h \mathrm{d}x \quad \forall v_h \in \mathbb{V}_h.
\end{equation}

\subsection{An auxiliary error estimate}
\label{sec:auxiliary_error_estimate}
In this section, we derive \EO{an auxiliary error estimate} that will be of fundamental importance to perform an a priori error analysis for the discretization \EO{introduced in Section \ref{sec:discrete_optimal_control_problem}.}

\begin{theorem}[\EO{an auxiliary error estimate: $d=2$}]\label{thm:auxiliary_error_estimate}
\EO{Let $\Omega \subset \mathbb{R}^2$ be an open, bounded, and convex polygonal domain. Let $a$ be such that \textnormal{\ref{A1} and \ref{A2}} hold. Assume, in addition, that $\partial a/ \partial y = \partial a/ \partial y(x,y)$ is Lipschitz in $y$ for a.e.~$x \in \Omega$ and that $a$ satisfies $a(\cdot,0) \in L^2(\Omega)$, $\partial a/ \partial y(\cdot,y) \in L^{\mathfrak{q}}(\Omega)$, for every $y \in \mathbb{R}$, where $\mathfrak{q}>2$, and that \eqref{eq:a_Lipschitz} holds with $\psi \in L^s(\Omega)$, where $s>2$. Let $\Omega_0$ and $\Omega_1$ be smooth subdomains such that $\mathcal{D} \subset \Omega_1 \Subset \Omega_0 \Subset \Omega$. If, in addition, $\psi, \mathsf{y}_d, \partial a/\partial y(\cdot,y) \in L^t(\Omega_0)$, for every $t < \infty$, then we have}
\begin{equation}
|(j'(\mathbf{u}) - j_h'(\mathbf{u})) \mathbf{v}| \lesssim h^2|\log h|^3 \| \mathbf{v} \|_{\mathbb{R}^{\ell}}.
\label{eq:error_estimate_j}
\end{equation}
\EO{for every $h \leq h_{\dagger}$. Here, $\mathbf{u}, \mathbf{v} \in \mathbb{R}^{\ell}$ and the hidden constant is independent of $h$.}
\end{theorem}
\begin{proof}
We begin the proof by noticing that 
\[
j'(\mathbf{u}) \mathbf{v} = \sum_{\mathsf{z} \in \mathcal{D}}( \mathsf{p}(\mathsf{z}) + \alpha \mathsf{u}_{\mathsf{z}}) \mathsf{v}_{\mathsf{z}},
\qquad
j_h'(\mathbf{u}) \mathbf{v} = \sum_{\mathsf{z} \in \mathcal{D}}( \hat{\mathsf{p}}_h(\mathsf{z}) + \alpha \mathsf{u}_{\mathsf{z}}) \mathsf{v}_{\mathsf{z}},
\]
where $\mathsf{p} \in H_0^1(\Omega) \cap H^2(\Omega)$ solves \eqref{eq:adj_eq} with $\mathsf{y} = \mathcal{S} \mathbf{u}$ and $\hat{\mathsf{p}}_h$ solves \eqref{eq:adjoint_discrete} with $\mathsf{y}$ being replaced by $\mathpzc{y}_h$, i.e., the solution to \eqref{eq:state_equation_discrete}. With these identities at hand, \EO{we obtain}
\begin{equation*}
|(j'(\mathbf{u}) - j_h'(\mathbf{u})) \mathbf{v}| 
 \leq
\sum_{\mathsf{z} \in \mathcal{D}} \left[ | \mathsf{p}(\mathsf{z}) - \mathpzc{p}_h(\mathsf{z})|
+
 |  \mathpzc{p}_h(\mathsf{z}) - \hat{\mathsf{p}}_h(\mathsf{z})| 
 \right] |\mathsf{v}_{\mathsf{z}}|
  =: \sum_{\mathsf{z} \in \mathcal{D}}
 \left[ \mathrm{I}_{\mathsf{z}} + \mathrm{II}_{\mathsf{z}} \right] |\mathsf{v}_{\mathsf{z}}|,
\end{equation*}
where $\mathpzc{p}_h$ denotes the solution to \eqref{eq:adjoint_discrete}.
Let $\mathsf{z} \in \mathcal{D}$. Invoke Theorem \ref{thm:error_estimate_adj} to arrive at \EO{$\mathrm{I}_{\mathsf{z}} \lesssim h^{2}|\log h|^2$.}
We now control $\mathrm{II}_{\mathsf{z}}$. To accomplish this task, we observe that $\mathpzc{p}_h - \hat{\mathsf{p}}_h \in \mathbb{V}_h$ is such that
\begin{multline*}
( \nabla (\mathpzc{p}_h -\hat{\mathsf{p}}_h) , \nabla v_h )_{L^2(\Omega)} 
+ 
\left( \tfrac{\partial a}{\partial y}(\cdot,\mathsf{y}) (\mathpzc{p}_h - \hat{\mathsf{p}}_h), v_h \right)_{L^2(\Omega)}
 \\
 = \left(\mathsf{y} - \mathpzc{y}_h, v_h\right)_{L^2(\Omega)} 
 +
 \left( \left( \tfrac{\partial a}{\partial y}(\cdot,\mathpzc{y}_h) - \tfrac{\partial a}{\partial y}(\cdot,\mathsf{y}) \right) \hat{\mathsf{p}}_h, v_h \right)_{L^2(\Omega)} \quad \forall v_h \in \mathbb{V}_h.
 \end{multline*}
We thus utilize an inverse inequality and a stability estimate for the previous problem to obtain
\begin{multline}
 \| \mathpzc{p}_h -\hat{\mathsf{p}}_h  \|^2_{L^{\infty}(\Omega)} 
 \lesssim 
 \mathfrak{i}_h^2 \| \nabla (\mathpzc{p}_h -\hat{\mathsf{p}}_h  )\|^2_{L^2(\Omega)} 
  \lesssim 
 \mathfrak{i}_h^2 \| \mathsf{y} - \mathpzc{y}_h  \|_{L^1(\Omega)}  \| \mathpzc{p}_h -\hat{\mathsf{p}}_h  \|_{L^{\infty}(\Omega)} 
 \\
 +  \mathfrak{i}_h^2  \left\| \left( \tfrac{\partial a}{\partial y}(\cdot,\mathpzc{y}_h) - \tfrac{\partial a}{\partial y}(\cdot,\mathsf{y}) \right) \hat{\mathsf{p}}_h \right\|_{L^1(\Omega)} \| \mathpzc{p}_h -\hat{\mathsf{p}}_h  \|_{L^{\infty}(\Omega)},
 \end{multline}
where $\mathfrak{i}_h = (1 + |\log h|)^{\frac{1}{2}}$ 
\cite[Lemma 4.9.2]{MR2373954}.
\EO{Invoke the fact that $\partial a/ \partial y = \partial a/ \partial y(x,y)$ is Lipschitz in $y$ for a.e.~$x \in \Omega$} and then the error estimate \eqref{eq:L1_error_estimate} to obtain
\[
 \EO{\mathrm{II}_{\mathsf{z}}}
 \lesssim
 \| \mathpzc{p}_h -\hat{\mathsf{p}}_h  \|_{L^{\infty}(\Omega)} \lesssim  \mathfrak{i}_h^2 \left(1 + \| \hat{\mathsf{p}}_h  \|_{L^{\infty}(\Omega} \right) \| \mathsf{y} - \mathpzc{y}_h  \|_{L^1(\Omega)} 
 \lesssim
 \mathfrak{i}_h^2 h^2 |\log h|^2.
\]
Collect the derived estimates for $\mathrm{I}_{\mathsf{z}}$ and $\mathrm{II}_{\mathsf{z}}$ to obtain \eqref{eq:error_estimate_j}. 
\end{proof}

\subsection{\EO{The discrete optimal control problem: convergence of discretizations}} 

We now provide a convergence result that, in essence, guarantees that a sequence of global solutions $\{ \bar{\mathbf{u}}_h \}_{h>0}$ of the discrete optimal control problems contains subsequences that converge, as $h \rightarrow 0$, to global solutions of the continuous optimal control problem.

\begin{theorem}[\EO{convergence to global solutions}]\label{thm:convergence_control}
Let $\Omega$ be an open, bounded, and convex polytopal domain. \EO{Let $a$ be such that \textnormal{\ref{A1}, \ref{A2},} and \eqref{eq:a_Lipschitz} hold. Let $h>0$ and let $\bar{\mathbf{u}}_h$ be a global solution of the discrete optimal control problem \eqref{eq:min_discrete}--\eqref{eq:state_equation_discret}. Then, there exist nonrelabeled subsequences $\{ \bar{\mathbf{u}}_h \}_{h>0}$ such that $\bar{\mathbf{u}}_h \rightarrow \bar{\mathbf{u}}$ in $\mathbb{R}^{\ell}$, as $h \rightarrow 0$, with $\bar{\mathbf{u}}$ being a global solution to the continuous optimal control problem \eqref{eq:min}--\eqref{eq:state_equation}. In addition, we have}
$
\lim _{h \rightarrow 0}  j_h( \bar{\mathbf{u}} _h ) =  j( \bar{\mathbf{u}} ).
$
\end{theorem}
\begin{proof}
Since $\{ \bar{\mathbf{u}}_h \}_{h>0} \subset \mathbf{U}_{ad}$ is uniformly bounded in $\mathbb{R}^{\ell}$, then there exists a nonrelabeled subsequence $\{ \bar{\mathbf{u}}_h \}_{h>0}$ such that $\bar{\mathbf{u}}_h \rightarrow \bar{\mathbf{u}}$ in $\mathbb{R}^{\ell}$ as $h \rightarrow 0$. Let us prove that $\bar{\mathbf{u}}$ is a global solution to \eqref{eq:min}--\eqref{eq:state_equation}. \EO{To accomplish this task, we} let $\tilde{\mathbf{u}} \in \mathbf{U}_{ad}$ be a global solution to \eqref{eq:min}--\eqref{eq:state_equation} and let $\{ \tilde{\mathbf{u}}_h \}_{h>0} \subset \mathbf{U}_{ad}$ be such that $\tilde{\mathbf{u}}_h \rightarrow \tilde{\mathbf{u}}$ in $\mathbb{R}^{\ell}$ as $h \rightarrow 0$. \EO{Since $\tilde{\mathbf{u}}$ is optimal for the continuous problem \eqref{eq:min}--\eqref{eq:state_equation} and, for every $h>0$, $\bar{\mathbf{u}}_h$ is a global solution to the discrete optimal control problem \eqref{eq:min_discrete}--\eqref{eq:state_equation_discret}, we obtain}
\begin{equation}
j(\tilde{\mathbf{u}}) \leq j(\bar{\mathbf{u}}) = 
\lim_{h \rightarrow 0} j_h(\bar{\mathbf{u}}_h) 
\leq  
\lim_{h \rightarrow 0} j_h(\tilde{\mathbf{u}}_h) 
=
j(\tilde{\mathbf{u}}).
\label{eq:several_inequalities}
\end{equation}
We observe that, since $\bar{\mathbf{u}}_h \rightarrow \bar{\mathbf{u}}$ and $\tilde{\mathbf{u}}_h \rightarrow \tilde{\mathbf{u}}$ in $\mathbb{R}^{\ell}$ as $h \rightarrow 0$, Theorem \ref{thm:convergence_state_equation} reveals that $\mathcal{S}_h \bar{\mathbf{u}}_h \rightarrow \mathcal{S} \bar{\mathbf{u}}$ and that $\mathcal{S}_h \tilde{\mathbf{u}}_h \rightarrow \mathcal{S} \tilde{\mathbf{u}}$ in $L^2(\Omega)$ as $h \rightarrow 0$. Consequently, it is immediate that $j_h(\bar{\mathbf{u}}_h) \rightarrow j(\bar{\mathbf{u}})$ and that $j_h(\tilde{\mathbf{u}}_h) \rightarrow j(\tilde{\mathbf{u}})$ as $h \rightarrow 0$. In view of \eqref{eq:several_inequalities}, we conclude that $\bar{\mathbf{u}}$ is a global solution to \eqref{eq:min}--\eqref{eq:state_equation}.
\end{proof}

\EO{We now provide a second convergence result: every strict local minimum of the continuous problem \eqref{eq:min}--\eqref{eq:state_equation} can be approximated by local minima of the discrete optimal control problems \eqref{eq:min_discrete}--\eqref{eq:state_equation_discret}.}

\begin{theorem}[\EO{convergence to strict local solutions}]\label{thm:convergence_control_2}
\EO{Let the assumptions of Theorem \ref{thm:convergence_control} hold. Let $\bar{\mathbf{u}}$ be a strict local minimum of \eqref{eq:min}--\eqref{eq:state_equation}. Then, there exist $h_{\nabla}>0$ and a sequence  of local minima $\{ \bar{\mathbf{u}}_h \}_{0<h \leq h_{\nabla}}$ of the discrete control problems such that $\bar{\mathbf{u}}_h \rightarrow \bar{\mathbf{u}}$ in $\mathbb{R}^{\ell}$ and 
$
j_h( \bar{\mathbf{u}} _h ) \rightarrow  j( \bar{\mathbf{u}} )
$
in $\mathbb{R}$ as $h \rightarrow 0$.}
\end{theorem}
\begin{proof}
Since $\bar{\mathbf{u}}$ is a strict local minimum of problem \eqref{eq:min}--\eqref{eq:state_equation}, there exists $\epsilon >0$ such that $\bar{\mathbf{u}}$ is the unique solution to $\min \{ j(\mathbf{u}): \mathbf{u} \in \mathbf{U}_{ad} \cap B_{\epsilon}(\bar{\mathbf{u}}) \}$, where $B_{\epsilon}(\bar{\mathbf{u}}):= \{ \mathbf{u} \in \mathbf{U}_{ad} : \| \mathbf{u} - \bar{\mathbf{u}} \|_{\mathbb{R}^{\ell}} \leq \epsilon \}$. \EO{Let us now introduce, for $h>0$, the discrete optimization problems:}
\begin{equation}
\label{eq:discrete_problems_Bepsilon}
\min \{ j_h(\mathbf{u}): \mathbf{u} \in \mathbf{U}_{ad} \cap B_{\epsilon}(\bar{\mathbf{u}}) \}.
\end{equation}
In view of the fact that $\mathbf{U}_{ad} \cap B_{\epsilon}(\bar{\mathbf{u}})$ is nonempty and compact, problem \eqref{eq:discrete_problems_Bepsilon} admits at least one solution. \EO{Let $h>0$ and let $\bar{\mathbf{u}}_h \in \mathbf{U}_{ad}$ be a solution to \eqref{eq:discrete_problems_Bepsilon}.} The arguments elaborated in the proof of Theorem \ref{thm:convergence_control} reveal the existence of nonrelabeled subsequences $\{ \bar{\mathbf{u}}_h \}_{h>0}$ that converge to solutions to $\min \{ j(\mathbf{u}): \mathbf{u} \in \mathbf{U}_{ad} \cap B_{\epsilon}(\bar{\mathbf{u}}) \}$. Since the latter minimization problem admits a unique solution, \EO{we must have that $\bar{\mathbf{u}}_h \rightarrow \bar{\mathbf{u}}$ in $\mathbb{R}^{\ell}$ as $h \rightarrow 0$ for the whole sequence. This, in particular, guarantees that the} constraint $\bar{\mathbf{u}}_h \in B_{\epsilon}(\bar{\mathbf{u}})$ is not active for $h$ sufficiently small. As a result, $\bar{\mathbf{u}}_h$ is a solution to the discrete optimal control problem \eqref{eq:min_discrete}--\eqref{eq:state_equation_discret}. This concludes the proof.
\end{proof}

\subsection{\EO{The discrete optimal control problem: error estimates}}
\label{sec:discrete_optimal_control_problem_error_estimates}
Let $\bar{\mathbf{u}}$ be a local minimum of the continuous optimal control problem and let $\{ \bar{\mathbf{u}}_h \}_{h>0}$ be a sequence of local minima of the discrete optimal control problems such that $\|  \bar{\mathbf{u}} - \bar{\mathbf{u}}_h  \|_{\mathbb{R}^{\ell}} \rightarrow 0$ as $h \rightarrow 0$; see Theorems \ref{thm:convergence_control} and \ref{thm:convergence_control_2}. In what follows, we derive an error estimate for the error $ \bar{\mathbf{u}} - \bar{\mathbf{u}}_h$ in $\mathbb{R}^{\ell}$. 

\begin{proposition}[\EO{an instrumental bound}]
\EO{Let $\Omega$ be an open, bounded, and convex polytopal domain. Let $a$ be such that \textnormal{\ref{A1}, \ref{A2}, \ref{A3}, and} \eqref{eq:a_Lipschitz} hold. Assume that, in addition, $\partial a/\partial y(\cdot,y) \in L^{\mathfrak{q}}(\Omega)$, for every $y \in \mathbb{R}$, where $\mathfrak{q}>2$.} If $\bar{\mathbf{u}} \in\mathbf{U}_{ad}$ is a local minimum of \eqref{eq:min}--\eqref{eq:state_equation} that satisfies the second order condition \eqref{eq:second_order_sufficient_condition}, or equivalently \eqref{eq:second_order_condition_equivalent}, then there exists $h_\Delta >0$ such that
\begin{equation}
\label{eq:instrumental_error_estimate}
\tfrac{\kappa}{2} \| \bar{\mathbf{u}} - \bar{\mathbf{u}}_h\|^2_{\mathbb{R}^{\ell}} \leq  \left( j'(\bar{\mathbf{u}}_h) - j'(\bar{\mathbf{u}}) \right) (\bar{\mathbf{u}}_h - \bar{\mathbf{u}})
\end{equation}
for every $h < h_{\Delta}$.
\label{pro:instrumental_bound}
\end{proposition}
\begin{proof}
We proceed on the basis of two steps.

\emph{Step 1.} We first prove that, for $h>0$ sufficiently small, $\bar{\mathbf{u}}_{h}-\bar{\mathbf{u}}$ belongs to $\mathbf{C}^{\tau}_{\bar{\mathbf{u}}}$ for some $\tau>0$. Since $\bar{\mathbf{u}}_{h}\in\mathbf{U}_{ad}$, it is immediate that $\bar{\mathbf{u}}_{h}-\bar{\mathbf{u}}$ satisfies the sign condition \eqref{eq:sign_cond}. It thus suffices to verify the remaining condition in \eqref{eq:cone_critical_tau}. To accomplish this task, we introduce the discrete variable $\Psi_{h}$ as follows:
\begin{equation*}
\Psi_{h}:= \{ \psi_{\mathsf{z},h} \}_{\mathsf{z}\in\mathcal{D}} \in\mathbb{R}^{\ell},
\qquad
\psi_{\mathsf{z},h}:=\bar{\mathsf{p}}_{h}(\mathsf{z})+\alpha \bar{\mathsf{u}}_{\mathsf{z},h}.
\end{equation*}
Since $\bar{\mathbf{u}}_{h}\rightarrow \bar{\mathbf{u}}$ in $\mathbb{R}^{\ell}$, the results of Theorem \ref{thm:convergence_state_equation} guarantee that $\bar{\mathsf{y}}_h \rightarrow \bar{\mathsf{y}}$ in $L^2(\Omega)$ \EO{as $h \rightarrow 0$},  which in turns implies that $\bar{\mathsf{p}}_h \rightarrow \bar{\mathsf{p}}$ in $C(\bar \Omega)$ \EO{as $h \rightarrow 0$}. We can thus deduce the existence of $h_{\dagger}>0$ such that
\begin{equation}\label{eq:tau_medios}
\|\Psi_{h}-\Psi\|_{\mathbb{R}^{\ell}} < \tau \quad \forall h\leq h_{\dagger},
\end{equation}
where, we recall that, $\Psi_{\mathsf{z}}$ is defined in \eqref{def:deriv_j}. Let $\mathsf{z}\in \mathcal{D}$ be arbitrary but fixed and assume that $\psi_{\mathsf{z}} > \tau > 0$. In view of the projection formula \eqref{eq:proj_formula}, we immediately conclude that $\bar{\mathsf{u}}_{\mathsf{z}}=\mathsf{a}_{\mathsf{z}}$. On the other hand, from \eqref{eq:tau_medios} we can obtain that $\psi_{\mathsf{z},h} > 0$ and thus that $\bar{\mathsf{u}}_{\mathsf{z},h}>-\alpha^{-1}\bar{\mathsf{p}}_{h}(\mathsf{z})$. 
This, on the basis of the projection formula \eqref{eq:proj_formula_discrete}, yields  $\bar{\mathsf{u}}_{\mathsf{z},h}=\mathsf{a}_{\mathsf{z}}$. Consequently, $\bar{\mathsf{u}}_{\mathsf{z}}=\bar{\mathsf{u}}_{\mathsf{z},h}=\mathsf{a}_{\mathsf{z}}$. Similar arguments allow us to obtain that, if $\psi_{\mathsf{z}} < -\tau < 0$, then $\bar{\mathsf{u}}_{\mathsf{z}}=\bar{\mathsf{u}}_{\mathsf{z},h}=\mathsf{b}_{\mathsf{z}}$. Since $\mathsf{z}\in\mathcal{D}$ is arbitrary, we can finally conclude that $\bar{\mathbf{u}}_{h}-\bar{\mathbf{u}}\in\mathbf{C}^{\tau}_{\bar{\mathbf{u}}}$.

\emph{Step 2.} Since $\bar{\mathbf{u}}_{h}-\bar{\mathbf{u}}\in\mathbf{C}^{\tau}_{\bar{\mathbf{u}}}$, with $\mathbf{C}_{\bar{\mathbf{u}}}^\tau$ defined in \eqref{eq:cone_critical_tau}, and $\bar{\mathbf{u}}$ satisfies \eqref{eq:second_order_sufficient_condition}, we are allowed to set $\mathbf{v}=\bar{\mathbf{u}}_{h}-\bar{\mathbf{u}}$ in \eqref{eq:second_order_condition_equivalent} to arrive at
\begin{equation}\label{eq:diff_1}
\kappa\|\bar{\mathbf{u}}_{h}-\bar{\mathbf{u}}\|_{\mathbb{R}^{\ell}}^2 \leq j''(\bar{\mathbf{u}})(\bar{\mathbf{u}}_{h}-\bar{\mathbf{u}})^2.
\end{equation}
On the other hand, in view of the mean value theorem we obtain, for some $\theta_{h} \in (0,1)$, 
\begin{equation*}
\label{eq:mean_value_identity}
(j'(\bar{\mathbf{u}}_{h})-j'(\bar{\mathbf{u}}))(\bar{\mathbf{u}}_{h}-\bar{\mathbf{u}})=j''(\hat{\mathbf{u}})(\bar{\mathbf{u}}_{h}-\bar{\mathbf{u}})^2,
\end{equation*}
where $\hat{\mathbf{u}}=\bar{\mathbf{u}}+\theta_{h}(\bar{\mathbf{u}}_{h}-\bar{\mathbf{u}})$. With \eqref{eq:diff_1} at hand, we can thus arrive at
\begin{align}\label{eq:ineq_u_tilde_bar}
\kappa\|\bar{\mathbf{u}}_{h}-\bar{\mathbf{u}}\|_{\mathbb{R}^{\ell}}^2 
& \leq (j'(\bar{\mathbf{u}}_{h})-j'(\bar{\mathbf{u}}))(\bar{\mathbf{u}}_{h}-\bar{\mathbf{u}}) + (j''(\bar{\mathbf{u}})-j''(\hat{\mathbf{u}}))(\bar{\mathbf{u}}_{h}-\bar{\mathbf{u}})^2.
\end{align}
Invoke the fact that $j''$ is continuous in $\mathbb{R}^{\ell}$, $\theta_{h} \in (0,1)$, and that $\bar{\mathbf{u}}_{h}\rightarrow \bar{\mathbf{u}}$ in $\mathbb{R}^{\ell}$, to deduce the existence of $h_{\ddagger}>0$ such that
\begin{equation*}
\left | (j''(\bar{\mathbf{u}})-j''(\hat{\mathbf{u}}))(\bar{\mathbf{u}}_{h}-\bar{\mathbf{u}})^2  \right |
\leq \frac{\kappa}{2}\|\bar{\mathbf{u}}_{h}-\bar{\mathbf{u}}\|_{\mathbb{R}^{\ell}}^2 \quad \forall h \leq h_{\ddagger}.
\end{equation*}
Replacing this inequality into \eqref{eq:ineq_u_tilde_bar} yields the desired inequality \eqref{eq:instrumental_error_estimate}. This concludes the proof.
\end{proof}

We conclude by presenting the following a priori error estimate for the approximation of an optimal control variable.

\begin{theorem}[\EO{a priori error estimate: $d=2$}]
\EOP{Let the assumptions of Theorem \ref{thm:auxiliary_error_estimate} and Proposition \ref{pro:instrumental_bound} hold. If $\bar{\mathbf{u}} \in\mathbf{U}_{ad}$ is a local minimum of \eqref{eq:min}--\eqref{eq:state_equation} that satisfies \eqref{eq:second_order_sufficient_condition}, then there exists $h_{\star} > 0$ such that
\begin{equation}
\label{eq:a_priori_error_estimate}
\|  \bar{\mathbf{u}} - \bar{\mathbf{u}}_h \|_{\mathbb{R}^{\ell}} \lesssim h^2|\log h|^3
\quad
\forall h < h_{\star},
\end{equation}
with a hidden constant that is independent of $h$.}
\label{thm:error_estimate_control_final}
\end{theorem}
\begin{proof}
Adding and subtracting the term $j_h'(\bar{\mathbf{u}}_h)(\bar{\mathbf{u}} - \bar{\mathbf{u}}_h)$ in the right hand side of \eqref{eq:instrumental_error_estimate} yields 
\[
\tfrac{\kappa}{2} \|  \bar{\mathbf{u}} - \bar{\mathbf{u}}_h \|^2_{\mathbb{R}^{\ell}} \leq  \left(  j'(\bar{\mathbf{u}}) - j_h'(\bar{\mathbf{u}}_h) \right) (\bar{\mathbf{u}}- \bar{\mathbf{u}}_h)
+
\left( j_h'(\bar{\mathbf{u}}_h) - j'(\bar{\mathbf{u}}_h)  \right) (\bar{\mathbf{u}}- \bar{\mathbf{u}}_h).
\]
We now invoke the continuous and discrete first order optimality conditions, \eqref{eq:variational_inequality} and \eqref{eq:first_order_optimality_condition_discrete}, respectively, to obtain $j'(\bar{\mathbf{u}})(\bar{\mathbf{u}}- \bar{\mathbf{u}}_h) \leq 0$ and $-j_h'(\bar{\mathbf{u}}_h)(\bar{\mathbf{u}}- \bar{\mathbf{u}}_h) \leq 0$. Consequently,
\[
\tfrac{\kappa}{2}  \| \bar{\mathbf{u}} - \bar{\mathbf{u}}_h \|^2_{\mathbb{R}^{\ell}} \leq \left( j_h'(\bar{\mathbf{u}}_h) - j'(\bar{\mathbf{u}}_h)  \right) (\bar{\mathbf{u}}- \bar{\mathbf{u}}_h).
\]
\EO{Utilize the auxiliary error estimate of Theorem \ref{thm:auxiliary_error_estimate} to immediately arrive at the desired bound \eqref{eq:a_priori_error_estimate}.}
\end{proof}

\begin{remark}[optimality]
\EO{The error estimate of Theorem \ref{thm:error_estimate_control_final} is nearly--optimal in terms of approximation (nearly because of the presence of the log-term).} 
\end{remark}

\section*{Acknowledgements}
\EO{We would like to thank the anonymous referees for several comments and suggestions that led to better results and an improved presentation. We would also like to thank Francisco Fuica for the careful reading of the manuscript and Eduardo Casas for fruitful discussions.} 

\section*{Funding}

\EO{Agencia Nacional de Investigaci\'on y Desarrollo (ANID) through FONDECYT 1220156 to E.O.}

\bibliographystyle{IMANUM-BIB}
\bibliography{semilinear_deltas}

\begin{thebibliography}{}

\bibitem[Adams \& Fournier(2003)Adams \& Fournier]{MR2424078}
{\sc Adams, R.~A. \& Fournier, J. J.~F.} (2003)
\newblock {\em Sobolev spaces\/}. Pure and Applied Mathematics (Amsterdam),
  vol. 140, second edn.
\newblock Elsevier/Academic Press, Amsterdam, pp. xiv+305.

\bibitem[Antil {\em et~al.}(2018)Antil, Ot\'{a}rola, \& Salgado]{MR3800041}
{\sc Antil, H., Ot\'{a}rola, E. \& Salgado, A.~J.} (2018)
\newblock Some applications of weighted norm inequalities to the error analysis
  of {PDE}-constrained optimization problems.
\newblock {\em IMA J. Numer. Anal.}, {\bf 38}, 852--883.

\bibitem[Behringer {\em et~al.}(2019)Behringer, Meidner, \& Vexler]{MR3973329}
{\sc Behringer, N., Meidner, D. \& Vexler, B.} (2019)
\newblock Finite element error estimates for optimal control problems with
  pointwise tracking.
\newblock {\em Pure Appl. Funct. Anal.}, {\bf 4}, 177--204.

\bibitem[Berm\'{u}dez {\em et~al.}(2004)Berm\'{u}dez, Gamallo, \&
  Rodr\'{\i}guez]{MR2086168}
{\sc Berm\'{u}dez, A., Gamallo, P. \& Rodr\'{\i}guez, R.} (2004)
\newblock Finite element methods in local active control of sound.
\newblock {\em SIAM J. Control Optim.}, {\bf 43}, 437--465.

\bibitem[Boccardo \& Gallou\"{e}t(1989)Boccardo \& Gallou\"{e}t]{MR1025884}
{\sc Boccardo, L. \& Gallou\"{e}t, T.} (1989)
\newblock Nonlinear elliptic and parabolic equations involving measure data.
\newblock {\em J. Funct. Anal.}, {\bf 87}, 149--169.

\bibitem[Brenner \& Scott(2008)Brenner \& Scott]{MR2373954}
{\sc Brenner, S.~C. \& Scott, L.~R.} (2008)
\newblock {\em The mathematical theory of finite element methods\/}. Texts in
  Applied Mathematics,  vol.~15, third edn.
\newblock Springer, New York, pp. xviii+397.

\bibitem[Casas(1985)Casas]{MR812624}
{\sc Casas, E.} (1985)
\newblock {$L^2$} estimates for the finite element method for the {D}irichlet
  problem with singular data.
\newblock {\em Numer. Math.}, {\bf 47}, 627--632.

\bibitem[Casas {\em et~al.}(2005)Casas, Mateos, \& Tr\"{o}ltzsch]{MR2150243}
{\sc Casas, E., Mateos, M. \& Tr\"{o}ltzsch, F.} (2005)
\newblock Error estimates for the numerical approximation of boundary
  semilinear elliptic control problems.
\newblock {\em Comput. Optim. Appl.}, {\bf 31}, 193--219.

\bibitem[Casas {\em et~al.}(2012)Casas, Clason, \& Kunisch]{MR2974716}
{\sc Casas, E., Clason, C. \& Kunisch, K.} (2012)
\newblock Approximation of elliptic control problems in measure spaces with
  sparse solutions.
\newblock {\em SIAM J. Control Optim.}, {\bf 50}, 1735--1752.

\bibitem[Casas \& Kunisch(2014)Casas \& Kunisch]{MR3162396}
{\sc Casas, E. \& Kunisch, K.} (2014)
\newblock Optimal control of semilinear elliptic equations in measure spaces.
\newblock {\em SIAM J. Control Optim.}, {\bf 52}, 339--364.

\bibitem[Casas \& Tr\"{o}ltzsch(2015)Casas \& Tr\"{o}ltzsch]{MR3311948}
{\sc Casas, E. \& Tr\"{o}ltzsch, F.} (2015)
\newblock Second order optimality conditions and their role in {PDE} control.
\newblock {\em Jahresber. Dtsch. Math.-Ver.}, {\bf 117}, 3--44.

\bibitem[Drelichman {\em et~al.}(2020)Drelichman, Dur\'{a}n, \&
  Ojea]{MR4060457}
{\sc Drelichman, I., Dur\'{a}n, R.~G. \& Ojea, I.} (2020)
\newblock A weighted setting for the numerical approximation of the {P}oisson
  problem with singular sources.
\newblock {\em SIAM J. Numer. Anal.}, {\bf 58}, 590--606.

\bibitem[Ern \& Guermond(2004)Ern \& Guermond]{Guermond-Ern}
{\sc Ern, A. \& Guermond, J.-L.} (2004)
\newblock {\em Theory and practice of finite elements\/}. Applied Mathematical
  Sciences,  vol. 159.
\newblock Springer-Verlag, New York, pp. xiv+524.

\bibitem[Evans \& Gariepy(1992)Evans \& Gariepy]{MR1158660}
{\sc Evans, L.~C. \& Gariepy, R.~F.} (1992)
\newblock {\em Measure theory and fine properties of functions\/}.
\newblock Studies in Advanced Mathematics.
\newblock CRC Press, Boca Raton, FL, pp. viii+268.

\bibitem[Folland(1999)Folland]{MR1681462}
{\sc Folland, G.~B.} (1999)
\newblock {\em Real analysis\/}.
\newblock Pure and Applied Mathematics (New York), second edn.
\newblock John Wiley \& Sons, Inc., New York, pp. xvi+386.

\bibitem[Fuica {\em et~al.}(2019)Fuica, Ot\'{a}rola, \& Salgado]{MR3945081}
{\sc Fuica, F., Ot\'{a}rola, E. \& Salgado, A.~J.} (2019)
\newblock An a posteriori error analysis of an elliptic optimal control problem
  in measure space.
\newblock {\em Comput. Math. Appl.}, {\bf 77}, 2659--2675.

\bibitem[Fuica {\em et~al.}(2021)Fuica, Ot\'{a}rola, \& Quero]{FOQ}
{\sc Fuica, F., Ot\'{a}rola, E. \& Quero, D.} (2021)
\newblock Error estimates for optimal control problems involving the {S}tokes
  system and {D}irac measures.
\newblock {\em Appl. Math. Optim.}, {\bf 84}, 1717--1750.

\bibitem[Gong {\em et~al.}(2014)Gong, Wang, \& Yan]{MR3225501}
{\sc Gong, W., Wang, G. \& Yan, N.} (2014)
\newblock Approximations of elliptic optimal control problems with controls
  acting on a lower dimensional manifold.
\newblock {\em SIAM J. Control Optim.}, {\bf 52}, 2008--2035.

\bibitem[Hern\'{a}ndez \& Ot\'{a}rola(2009)Hern\'{a}ndez \&
  Ot\'{a}rola]{MR2525606}
{\sc Hern\'{a}ndez, E. \& Ot\'{a}rola, E.} (2009)
\newblock A locking-free {FEM} in active vibration control of a {T}imoshenko
  beam.
\newblock {\em SIAM J. Numer. Anal.}, {\bf 47}, 2432--2454.

\bibitem[Jerison \& Kenig(1995)Jerison \& Kenig]{MR1331981}
{\sc Jerison, D. \& Kenig, C.~E.} (1995)
\newblock The inhomogeneous {D}irichlet problem in {L}ipschitz domains.
\newblock {\em J. Funct. Anal.}, {\bf 130}, 161--219.

\bibitem[Leykekhman \& Vexler(2013)Leykekhman \& Vexler]{MR3116646}
{\sc Leykekhman, D. \& Vexler, B.} (2013)
\newblock Optimal a priori error estimates of parabolic optimal control
  problems with pointwise control.
\newblock {\em SIAM J. Numer. Anal.}, {\bf 51}, 2797--2821.

\bibitem[Luenberger(2003)Luenberger]{MR2012832}
{\sc Luenberger, D.~G.} (2003)
\newblock {\em Linear and nonlinear programming\/}, second edn.
\newblock Kluwer Academic Publishers, Boston, MA, pp. xviii+491.

\bibitem[Pieper \& Vexler(2013)Pieper \& Vexler]{MR3072225}
{\sc Pieper, K. \& Vexler, B.} (2013)
\newblock A priori error analysis for discretization of sparse elliptic optimal
  control problems in measure space.
\newblock {\em SIAM J. Control Optim.}, {\bf 51}, 2788--2808.

\bibitem[Schatz \& Wahlbin(1977)Schatz \& Wahlbin]{MR431753}
{\sc Schatz, A.~H. \& Wahlbin, L.~B.} (1977)
\newblock Interior maximum norm estimates for finite element methods.
\newblock {\em Math. Comp.}, {\bf 31}, 414--442.

\bibitem[Schatz \& Wahlbin(1982)Schatz \& Wahlbin]{MR637283}
{\sc Schatz, A.~H. \& Wahlbin, L.~B.} (1982)
\newblock On the quasi-optimality in {$L_{\infty }$} of the {$\dot
  H^{1}$}-projection into finite element spaces.
\newblock {\em Math. Comp.}, {\bf 38}, 1--22.

\bibitem[Stampacchia(1965)Stampacchia]{MR192177}
{\sc Stampacchia, G.} (1965)
\newblock Le probl\`eme de {D}irichlet pour les \'{e}quations elliptiques du
  second ordre \`a coefficients discontinus.
\newblock {\em Ann. Inst. Fourier (Grenoble)\/}, {\bf 15}, 189--258.

\bibitem[Tr\"oltzsch(2010)Tr\"oltzsch]{Troltzsch}
{\sc Tr\"oltzsch, F.} (2010)
\newblock {\em Optimal control of partial differential equations\/}. Graduate
  Studies in Mathematics,  vol. 112.
\newblock American Mathematical Society, Providence, RI, pp. xvi+399.

\end{thebibliography}

\end{document}